\documentclass[10pt,a4paper]{amsart}

\usepackage[utf8]{inputenc}
\usepackage[english]{babel}
\usepackage[foot]{amsaddr}
\usepackage{amsmath,amsfonts,amssymb,amsthm}
\usepackage{mathtools} %:=
\usepackage{dsfont} %1
\usepackage{bm} %bold italic
\usepackage{tikz-cd}
\usepackage{adjustbox} %alignment in a tabular with graphics

\newtheorem{thm}{Theorem}
\newtheorem{prop}[thm]{Proposition}
\newtheorem{cor}[thm]{Corollary}
\newtheorem{lem}[thm]{Lemma}

\theoremstyle{definition}
\newtheorem{df}[thm]{Definition}

\theoremstyle{remark}
\newtheorem{rem}[thm]{Remark}

\DeclareMathOperator{\Mod}{Mod}
\DeclareMathOperator{\Aut}{Aut}
\DeclareMathOperator{\PD}{PD}
\DeclareMathOperator{\GEM}{GEM}
\DeclareMathOperator{\Beta}{Beta}
\DeclareMathOperator{\area}{area}

\DeclareMathOperator{\Poisson}{Poi}
\DeclareMathOperator{\id}{id}

\newcommand{\NN}{\mathbb{N}}
\newcommand{\ZZ}{\mathbb{Z}}
\newcommand{\QQ}{\mathbb{Q}}
\newcommand{\RR}{\mathbb{R}}
\newcommand{\CC}{\mathbb{C}}
\newcommand{\EE}{\mathbb{E}}
\newcommand{\PP}{\mathbb{P}}

\newcommand{\cG}{\mathcal{G}}
\newcommand{\cL}{\mathcal{L}}
\newcommand{\cM}{\mathcal{M}}
\newcommand{\cML}{\mathcal{ML}}
\newcommand{\cQ}{\mathcal{Q}}

\newcommand{\cY}{\mathcal{Y}}
\newcommand{\cZ}{\mathcal{Z}}

\newcommand{\Ell}{\boldsymbol\ell}
\newcommand{\Multiplicity}{\operatorname{\mathbf{mult}}}
\newcommand{\multiplicity}{\operatorname{mult}}
\newcommand{\da}{{\mkern1mu \downarrow}}
\newcommand{\height}{\operatorname{height}}

\let\Gamma\varGamma
\let\varGamma\temp

\let\Sigma\varSigma
\let\varSigma\temp

\title[Length partition of random multicurves on hyperbolic surfaces]{Length partition of random multicurves \\ on large genus hyperbolic surfaces}
\date{\today}

\author{Vincent Delecroix}
\address{Univ. Bordeaux, CNRS, Bordeaux INP, LaBRI, UMR 5800, F-33400 Talence, France}
\email{vincent.delecroix@u-bordeaux.fr}

\author{Mingkun Liu}
\address{Université de Paris and Sorbonne Université, CNRS, IMJ-PRG, F-75006 Paris, France}
\email{mingkun.liu@imj-prg.fr}

\begin{document}
\maketitle

\begin{abstract}
We study the length statistics of the components of a random multicurve on a surface of genus $g \geq 2$. For each fixed genus, the existence of such statistics follows from the work of M.~Mirzakhani~\cite{Mir16}, F.~Arana-Herrera~\cite{AH21} and M.~Liu~\cite{Liu19}. We prove that as the genus $g$ tends to infinity the statistics converge in law to the Poisson--Dirichlet distribution of parameter $\theta=1/2$. In particular, as the genus tends to infinity the mean length of the three longest components converge respectively to $75.8\%$, $17.1\%$ and $4.9\%$ of the total length.
\end{abstract}

%%%%%%%%%%%%%%%%%%%%%%%%%%%%%%%%%%%%%%%%%%%%%%%%%%%%%%%%%%%%%%%%%%%%%%%%%%%%%%%%%
\section{Introduction}
\subsection{Lengths statistics of random multicurves in large genus}
Let $X$ be a closed Riemann surface of genus $g \geq 2$ endowed with its conformal hyperbolic metric of constant curvature $-1$. A \emph{simple closed curve} on $X$ is a connected closed curve on $X$, non-homotopic to a point and without self-intersection. In the free homotopy class of a simple closed curve $\gamma$, there exists a unique geodesic representative with respect to $X$. We denote by $\ell_X(\gamma)$ the length of this geodesic representative.

A \emph{multicurve} on $X$ is a multiset of disjoint simple closed curves on $X$. Given a multicurve $\gamma$, a \emph{component} of $\gamma$ is a maximal family of freely homotopic curves in $\gamma$. The cardinal of a component is called its \emph{multiplicity} and the \emph{length} of a component is the sum of the lengths of the simple curves belonging to the component (or equivalently its multiplicity multiplied by the length of any simple closed curve if the component). A multicurve is called \emph{primitive} if all its components have multiplicity one. We denote by $\Ell^\da_X(\gamma)$ the vector of the lengths of each component sorted in decreasing order and by $\Multiplicity(\gamma)$ the multiset of the multiplicities of the components of $\gamma$, and by $\multiplicity(\gamma)$ the maximum of $\Multiplicity(\gamma)$. Neither $\Multiplicity(\gamma)$ nor $\multiplicity(\gamma)$ depend on the hyperbolic structure $X$. We define $\ell_X(\gamma)$ as the sum of the entries of $\Ell^\da_X(\gamma)$ and the \emph{normalized length vector} to be
\[
\hat{\Ell}_X^\da(\gamma) \coloneqq \frac{\Ell^\da_X(\gamma)}{\ell_X(\gamma)}.
\]
We denote by $\cML_X(\ZZ)$ the set of homotopy classes of multicurves on $X$. Our notation for the set of multicurves is explained by the fact that multicurves are the integral points of the space of measured laminations usually denoted $\cML_X$.

In order to make sense of convergence, we need all normalized vectors to belong to the same space. For an integer $k \geq 1$ and
a real number $r > 0$ let us define
\[
\Delta_{\leq r}^k
\coloneqq
\{(x_1,x_2,\ldots,x_k) \in [0,\infty)^k : x_1 + x_2 + \cdots + x_k \leq r\}.
\]
Let also define
\[
\Delta_{\leq r}^\infty
\coloneqq
\{(x_1,x_2,\dots) \in [0,\infty)^{\NN} : x_1 + x_2 + \cdots \leq r\}.
\]
For $k \leq k'$ we have an injection $\Delta_{\leq r}^k \to \Delta^{k'}_{\leq r}$ by completing vectors with zeros.
The infinite simplex $\Delta_{\leq r}^\infty$ is the inductive limit of these injections and we always identify $\Delta_{\leq r}^k$ as a subspace of $\Delta_{\leq r}^\infty$. In particular each vector $\hat{\Ell}_X^\da(\gamma)$ is naturally an element of $\Delta_{\leq 1}^\infty$ by completing its coordinates with infinitely many zeros.

As our aim is to study convergence of random infinite vectors, let us mention that $\Delta_{\leq 1}^\infty$ is a closed subset of $[0,1]^\NN$ endowed with the product topology. This topology coincides with the topology of the inductive limit. When we consider a convergence in distribution on $\Delta_{\leq 1}^\infty$ we mean convergence in the space of Borel probability measures on $\Delta_{\leq 1}^\infty$ which is a compact set.

The following result is a consequence of the works~\cite{AH21} and~\cite{Liu19}.
\begin{thm} \label{thm:Lg}
Let $g \geq 2$ and $m \in \NN \cup \{+\infty\}$. There exists a random variable $L^{(g,m)\downarrow} = (L^{(g,m)\downarrow}_1, L^{(g,m)\downarrow}_2,\ldots)$ on $\Delta^{3g-3}_{\leq 1}$ with the following properties. For
any Riemann surface $X$ of genus $g$, as $R \to \infty$ we have the following convergence in distribution
\[
\frac{1}{s_X(R, m)} \sum_{\substack{\gamma \in \cML_X(\ZZ) \\ \ell_X(\gamma)\leq R \\ \multiplicity(\gamma) \le m }} \delta_{\hat{\Ell}_X^\da(\gamma)} \xrightarrow[R \to \infty]{} L^{(g,m)\downarrow}
\]
where $\delta_{\hat{\Ell}_X^\da(\gamma)}$ is the Dirac mass at the vector $\hat{\Ell}_X^\da(\gamma)$ and $s_X(R, m) \coloneqq \# \{\gamma \in \cML_X(\ZZ) : \ell_X(\gamma)\leq R  \text{ and } \multiplicity(\gamma) \le m \}$ is the number of multicurves on $X$ of length at most $R$ and multiplicities at most $m$.
\end{thm}
We actually prove a more precise version of the above statement, Theorem~\ref{thm:LgMorePrecise}, in which the law of $L^{(g,m)\downarrow}$ is made explicit. Remark that the limit depends only on the genus of $X$ and not on its hyperbolic metric.

\smallskip

The Poisson--Dirichlet distribution is a probability measure on $\Delta_{\leq 1}^\infty$. The simplest way to introduce it is via the \emph{stick-breaking process}. Let $U_1,U_2,\dots$, be i.i.d.\ random variables with law $\Beta(1,\theta)$ (i.e.\ they are supported on $]0,1]$ with density $\theta (1-x)^{\theta -1}$). Define the vector
\[
    V \coloneqq (U_1, (1-U_1)\, U_2, (1-U_1)(1-U_2)\,U_3, \ldots).
\]
Informally, the components of $V$ are obtained by starting from a stick of length 1 identified with $[0,1]$. At the first stage, $U_1$ determines where we break the first piece and we are left with a stick of size $1-U_1$. We then repeat the process ad libitum.
The law of $V$ is the \emph{Griffiths-Engen-McCloskey distribution of parameter $\theta$} that we denote $\GEM(\theta)$. The \emph{Poisson--Dirichlet distribution of parameter $\theta$}, denoted $\PD(\theta)$, is the distribution of $V^\downarrow$, the vector $V$ whose entries are sorted in decreasing order. For more details, we refer the reader to Section~\ref{ssec:PoissonDirichletAndGEM}. The distribution $\PD(1)$ is the limit distribution of the orbit length of uniform random permutations. The distribution $\PD(\theta)$ appears when considering the Ewens distribution with parameter $\theta$ on the symmetric group. See Section~\ref{sssec:permutations} below for a more detailed discussion on permutations.

Our main result is the following.
\begin{thm}\label{thm:PD}
For any $m \in \NN \cup \{+\infty\}$, the sequence $(L^{(g,m)\downarrow})_{g \ge 2}$ converges in distribution to $\PD(1/2)$ as $g \to \infty$.
\end{thm}
The most interesting cases of this convergence are for $m=1$ (primitive multicurves) and $m=\infty$ (all multicurves). Let us insist that $L^{(g,1)\da}$ and $L^{(g,+\infty)}$ converge to the same limit as $g \to \infty$.

All marginals of the Poisson--Dirichlet law can be computed, see for example~\cite[Section~4.11]{ABT03}. In particular if $V = (V_1, V_2, \ldots) \sim \PD(\theta)$ then
\[
    \EE((V_j)^n)
    =
    \frac{\varGamma(\theta + 1)}{\varGamma(\theta + n)} \int_0^\infty \frac{(\theta E_1(x))^{j-1}}{(j-1)!} \, x^{n-1} e^{-x-\theta E_1(x)} \, dx
\]
where $E_1(x) \coloneqq\int_x^\infty \frac{e^{-y}}{y}\, dy$. The formulas can be turned into a computer program and values were tabulated in~\cite{Gri79,Gri88}. For $\theta=1/2$ we have
\[
    \EE(V_1) \approx 0.758,\quad
    \EE(V_2) \approx 0.171,\quad \text{and} \quad
    \EE(V_3) \approx 0.049.
\]

\subsection{Further remarks}
\subsubsection{Square-tiled surfaces}
In this section we give an alternative statement of Theorem~\ref{thm:PD} in terms of square-tiled surfaces. The correspondence between statistics of multicurves and statistics of square-tiled surfaces is developed in~\cite{DGZZ21} and~\cite{AH20b} and we refer the readers to these two references.

A \emph{square-tiled surface} is a connected surface obtained from gluing finitely many unit squares $[0,1] \times [0,1]$ along their edges by translation $z \mapsto z + u$ or ``half-translation'' $z \mapsto -z + u$. Combinatorially, one can label the squares from $1$ to $N$ and then a square-tiled surface is encoded by two involutions without fixed points $(\sigma, \tau)$ of $\{\pm 1, \pm 2, \ldots, \pm N\}$. More precisely, $\sigma$ encodes the horizontal gluings: $+i$ and $-i$ are respectively the right and left sides of the $i$-th squares. The orbits of $\sigma$ with different signs are glued by translations and the ones with same signs are glued by half-translations. And $\tau$ encodes the vertical gluings : $+i$ and $-i$ are respectively the top and bottom sides of the $i$-th squares. The labelling is irrelevant in our definition and two pairs $(\sigma, \tau)$ and $(\sigma', \tau')$ encode the same square-tiled surface if there exists a permutation $\alpha$ of $\{\pm 1, \pm 2, \ldots, \pm N\}$ so that $\alpha(-i) = - \alpha(+i)$, $\sigma' = \alpha \circ \sigma \circ \alpha^{-1}$ and $\tau' = \alpha \circ \tau \circ \alpha^{-1}$.

A square-tiled surface comes with a conformal structure and a quadratic form coming from the conformal structure of the unit square and the quadratic form $dz^2$ (both are being preserved by translations and half-translations). This quadratic form might have simple poles and we denote by $\cQ_g(\ZZ)$ the set of holomorphic square-tiled surfaces of genus $g$.

A square-tiled surface come equipped with a filling pair of multicurves $(\gamma_h, \gamma_v)$ coming respectively from the gluings of the horizontal segments $[0,1] \times \{1/2\}$ and vertical segments $\{1/2\} \times [0,1]$ of each square. Conversely, the dual graph of a filling pair of multicurves in a surface of genus $g$ defines a square-tiled surface in $\cQ_g(\ZZ)$. Our notation comes from the fact that holomorphic square-tiled surfaces can be seen as integral points in the moduli space of quadratic differentials $\cQ_g$.
A component of the multicurve $\gamma_h$ corresponds geometrically to a horizontal cylinder. For a square-tiled surface $M$ we denote by $\boldsymbol{A}^\downarrow(M)$ the normalized vector of areas of these horizontal cylinders sorted in decreasing order and by $\height(M)$ the maximum of their heights. Here as in the introduction, normalized mean that we divide by the sum of entries of a vector which coincides with $\area(M)$. The following is a particular case of~\cite[Theorem~1.29]{DGZZ21} using the explicit formulas for $L^{(g,m)}$ given in Theorem~\ref{thm:LgMorePrecise}.
\begin{thm}[\cite{DGZZ21}] \label{thm:LgSquareTiled}
Let $g \ge 2$ and $m \in \NN \cup \{+\infty\}$. Let $L^{(g,m)}$ be the random variable from Theorem~\ref{thm:Lg}. Then as $N \to \infty$ we have the following convergence in distribution
\[
    \frac{1}{
        \#\left\{
            M \in \cQ_g(\ZZ) :
            \begin{array}{l}
                \height(M) \le m \\
                \area(M) \le N
            \end{array}
        \right\}
    }
    \sum_{\substack{M \in \cQ_g(\ZZ) \\ \height(M) \le m \\ \area(M) \le N}}
    \delta_{\boldsymbol{A}^\downarrow(M)}
    \to
    L^{(g,m)}.
\]
\end{thm}
An important difference to notice between Theorem~\ref{thm:Lg} and Theorem~\ref{thm:LgSquareTiled} is that in the former the (hyperbolic) metric $X$ is fixed and we sum over the multicurves $\gamma$ while in the latter we sum over the discrete set of holomorphic square-tiled surfaces $M$.

Using Theorem~\ref{thm:LgSquareTiled}, our Theorem~\ref{thm:PD} admits the following reformulation.
\begin{cor}
The vector of normalized areas of horizontal cylinders of a random square-tiled surface of genus $g$ converges in distribution to $\PD(1/2)$ as $g$ tends to $\infty$.
\end{cor}

\subsubsection{Permutations and multicurves}
\label{sssec:permutations}
Given a permutation $\sigma$ in $S_n$ we denote by $K_n(\sigma)$ the number of orbits it has on $\{1, 2, \ldots, n\}$ or equivalently the number of cycles in its disjoint cycle decomposition. The \emph{Ewens measure with parameter $\theta$} on $S_n$ is the probability measure defined by
\[
\PP_{n,\theta}(\sigma) \coloneqq \frac{\theta^{K_n(\sigma)}}{Z_{n,\theta}}
\quad
\text{where}
\quad
Z_{n,\theta} \coloneqq \sum_{\sigma \in S_n} \theta^{K_n(\sigma)}.
\]
Then under $\PP_{n,\theta}$, as $n \to \infty$ we have that
\begin{itemize}
\item the random variable $K_n$ behaves as a Poisson distribution $\Poisson(\theta \log(n))$ (e.g.\ by mean of a local limit theorem),
\item the normalized sorted vector of cycle lengths of $\sigma$ tends to $\PD(\theta)$,
\item the number of cycles of length $k$ of $\sigma$ converges to $\Poisson(\theta/k)$.
\end{itemize}
See for example~\cite{ABT03}.

By analogy let us denote by $K^{(g,m)}$ the number of non-zero components of $L^{(g,m)}$. In~\cite{DGZZ20}, it is proven that $K^{(g,m)}$ behaves as a Poisson distribution with parameter $\frac{\log(g)}{2}$ (by mean of a local limit theorem) independently of $m$. In other words, it behaves as the number of cycles $K_g(\sigma)$ for a random permutation $\sigma$ under $\PP_{g,1/2}$.

Our Theorem~\ref{thm:PD} provides another connection between $L^{(g,m)}$ and $\PP_{g,1/2}$. Namely, $L^{(g,m)\da}$ is asymptotically close to the normalized sorted vector of the cycle length of $\sigma$ under $\PP_{g,1/2}$.

Finally, let us mention that components of $L^{(g,m)}$ of the order of $o(1)$ are invisible in the convergence towards $\PD(1/2)$. It is a consequence of Theorem~\ref{thm:PD} that the macroscopic components of the order of a constant carry the total mass. Building on the intuition that in the large genus asymptotic regime random multicurves on a surface $X$ of genus $g$ behave like the cycles of a random permutation in the symmetric group $S_g$, one should expect to have a Poisson limit for components of order $g^{-1}$ and that there is no component of order $g^{-1-\epsilon}$. In a work in progress, we provide an affirmative answer to this intuition. However, because lengths are continuous parameters, the limit is a continuous Poisson process and not a discrete one supported on $\NN$ as in the permutation case.

\subsection{Proof overview and structure of the paper}
The first step of the proof consists in writing an explicit expression for the random variable $L^{(g,m)\downarrow}$ that appears in Theorem~\ref{thm:Lg}, see Theorem~\ref{thm:LgMorePrecise} in Section~\ref{sec:RandomMulticurves}. The formula follows from the work of M.~Mirzakhani on pants decompositions~\cite{Mir08} and the result of F.~Arana-Herrera~\cite{AH21} and M.~Liu~\cite{Liu19} on length distribution for each fixed topological type of multicurves. The expression of $L^{(g,m)\downarrow}$ can be seen as a refinement of the formula for the Masur--Veech volume of the moduli space of quadratic differentials from~\cite{DGZZ21}.

The formula for $L^{(g,m)\downarrow}$ involves a super-exponential number of terms in $g$ (one term for each topological type of multicurve on a surface of genus $g$). However, in the large genus limit only $O(\log(g))$ terms contribute. This allows us to consider a simpler random variable $\tilde{L}^{(g,m,\kappa)\downarrow}$ which, asymptotically, coincides with $L^{(g,m)\downarrow}$. See Theorem~\ref{thm:reduction} in Section~\ref{sec:reduction}. This reduction is very similar to the one used for the large genus asymptotics of Masur--Veech volumes in~\cite{Agg21} and~\cite{DGZZ20}.

The core of our proof consists in proving the convergence of moments of the simpler variable $\tilde{L}^{(g,m,\kappa)\downarrow}$. We do not use directly $\tilde{L}^{(g,m,\kappa)\downarrow}$ but its size-biased version $\tilde{L}^{(g,m,\kappa)*}$. The definition of size bias and the link with the Poisson--Dirichlet distribution is explained in Section~\ref{sec:PDandGEM}. In Section~\ref{sec:LargeGenus}, we show that the moments $\tilde{L}^{(g,m,\kappa)*}$ converge to the moments of $\GEM(1/2)$ which is the size-biased version of the Poisson--Dirichlet process $\PD(1/2)$, see Theorem~\ref{thm:GEM}.

\subsection{Acknowledgement}
We warmly thank Anton Zorich who encouraged us to join our forces and knowledge from~\cite{Liu19} and~\cite{DGZZ20} to study the lengths statistics of random multicurves. The second author would like to thank Gr\'egoire Sergeant-Perthuis and Maud Szusterman for helpful conversations about probability theory.

The work of the first named author is partially supported by the ANR-19-CE40-0003 grant.

\section{Background material}
In this section we introduce notations and state results from the literature that are used in our proof.

\subsection{Multicurves and stable graphs}
\label{ssec:multicurvesAndStableGraphs}
Recall from the introduction that a multicurve on a hyperbolic surface $X$ of genus $g$ is a finite multiset of free homotopy classes of disjoint simple closed curves. We denote by $\cML_X(\ZZ)$ the set of multicurves on $X$. The homotopy classes that appear in a multicurve $\gamma$ are called components. There are at most $3g-3$ of them. The multiplicity of a component is the number of times it is repeated in $\gamma$, and $\gamma$ is primitive if all the multiplicities are $1$.

Let us also recall our notations:
\begin{itemize}
\item $\ell_X(\gamma)$: total length of $\gamma$,
\item $\Ell^\da_X(\gamma)$: length vector of the components of $\gamma$,
\item $\multiplicity(\gamma)$: maximum multiplicity of component in $\gamma$,
\item $\Multiplicity(\gamma)$: multiset of multiplicities of components in $\gamma$.
\end{itemize}

The mapping class group $\Mod(X)$ of $X$ acts on multicurves. We call \emph{topological type} of a multicurve its equivalence class under the $\Mod(X)$-action. For each fixed genus $g$, there are finitely many topological types of primitive multicurves and countably many topological types of multicurves. They are conveniently encoded by respectively stable graphs and weighted stable graphs that we define next. Informally given a multicurve $\gamma$ with components $\gamma_1,\dots,\gamma_k$ and multiplicities $m_1,\dots,m_k$ we build a dual graph $\Gamma$ as follows:
\begin{itemize}
\item we add a vertex for each connected component of the complement $X \smallsetminus (\gamma_1 \cup \cdots \cup \gamma_k)$; the vertex $v$ carries an integer weight the genus $g_v$ of the corresponding component,
\item we add an edge for each component $\gamma_i$ of the multicurve between the two vertices corresponding to the connected components bounded by $\gamma_i$; this edge carries a weight $m_i$.
\end{itemize}

More formally, a \emph{stable graph} $\Gamma$ is a 5-tuple $(V, H, \iota, \sigma, \{g_v\}_{v \in V})$ where
\begin{itemize}
\item $V$ is a finite set called \emph{vertices},
\item $H$ is a finite set called \emph{half-edges},
\item $\iota : H \to H$ is an involution without fixed points on $H$; each pair $\{h, \iota(h)\}$ is called an \emph{edge}
and we denote by $E(\Gamma)$ the set of edges,
\item $\sigma: H \to V$ is a surjective map ($\sigma(h)$ is the vertex at which $h$ is rooted),
\item $g_v \in \ZZ_{\geq 0}$,
\end{itemize}
such that
\begin{itemize}
\item (\emph{connectedness}) for each pair of vertices $u,v \in V$ there exists
a sequence of edges, $\{h_1, h'_1\}$, $\{h_2, h'_2\}$, \ldots, $\{h_n, h'_n\}$
such that $\sigma(h_1) = u$, $\sigma(h'_n) = v$ and for $i \in \{1, \ldots, n-1\}$
we have $\sigma(h'_i) = \sigma(h_{i+1})$,
\item (\emph{stability}) for each vertex $v\in V$ we have
\[
    2g_v - 2 + \deg(v) > 0
\]
where $\deg(v) \coloneqq |\sigma^{-1}(v)|$ is the \emph{degree} of the vertex $v$.
\end{itemize}
Given a stable graph $\Gamma$, its \emph{genus} is
\[
    g(\Gamma)
    \coloneqq
    |E| - |V| + 1 + \sum_{v\in V} g(v).
\]
An \emph{isomorphism} between two stable graphs $\Gamma=(V, H, \iota, \sigma, g)$ and $\Gamma' = (V', H', \iota', \sigma', g')$
is a pair of bijections $\phi: V \to V'$ and $\psi: H \to H'$ such that
\begin{itemize}
\item $\psi \circ \iota = \iota' \circ \psi$ (in other words, $\psi$ maps an edge to an edge)
\item $\phi \circ \sigma = \sigma' \circ \psi$,
\item for each $v \in V$, we have $g'_{\phi(v)} = g_v$.
\end{itemize}
Note that $\psi$ determines $\phi$ but it is convenient to record automorphism as a pair $(\phi, \psi)$.
We denote by $\Aut(\Gamma)$ the set of automorphisms of $\Gamma$ and by $\cG_g$ the finite set of isomorphism classes of stable graphs of genus $g$.

A \emph{weighted stable graph} is a pair $(\Gamma, \boldsymbol{m})$ where $\Gamma$ is a stable graph and $\boldsymbol{m} \in\NN^{E(\Gamma)}$.
An \emph{isomorphism} between two weighted stable graphs $(\Gamma, \boldsymbol{m})$ and $(\Gamma', \boldsymbol{m}')$ is
an isomorphism $(\phi, \psi)$ between $\Gamma$ and $\Gamma'$ such that for each edge $e$ of $\Gamma'$
we have $\boldsymbol{m}_{e} = \boldsymbol{m}'_{\psi(e)}$ (where we use $\psi(e)$ to denote $\{\psi(h), \psi(h')\}$
for the edge $e = \{h,h'\} \subset H'$). We denote by $\Aut(\Gamma, \boldsymbol{m})$ the set of
automorphisms of the weighted graph $(\Gamma, \boldsymbol{m})$. There is a one-to-one correspondence between
topological types of multicurves and weighted stable graphs. Primitive multicurves correspond to the case where
all edges carry weight $1$.

\subsection{$\psi$-classes and Kontsevich polynomial}
\label{subsec:PsiClasses}
The formula for the random variable $L^{(g,m)\downarrow}$ that appears in Theorem~\ref{thm:Lg} involves intersection numbers of $\psi$-classes that we introduce now. These rational numbers are famously related to the Witten conjecture~\cite{Wit91} proven by Kontsevich~\cite{Kon92}.

Let $\overline{\cM}_{g,n}$ denote the Deligne--Mumford compactification of moduli space of smooth complex curves of genus $g$ with $n$ marked points.
There exist $n$ so-called \emph{tautological line bundles} $\cL_1,\dots,\cL_n \to \overline{\cM}_{g,n}$ over $\overline{\cM}_{g,n}$ such that the fiber of $\cL_i$ at $(C;x_1,\dots,x_n) \in \overline{\cM}_{g,n}$ is the cotangent space of $C$ at the $i$-th marked point $x_i$. The $i$-th \emph{psi-class} $\psi_i$ is defined as the first Chern class of the $i$-th tautological line bundle $c_1(\cL_i)\in H^2(\overline{\cM}_{g,n},\QQ)$. We use the following standard notation
\[
    \langle\tau_{d_1} \cdots \tau_{d_n}\rangle_{g,n}
    \coloneqq
    \int_{\overline{\cM}_{g,n}} \psi_1^{d_1} \cdots \psi_n^{d_n}
\]
when $d_1+\cdots +d_n = \dim_\CC\overline{\cM}_{g,n} = 3g-3+n$. All
these intersection numbers are positive rational numbers and can be
computed by recursive equations from $\langle \tau_0^3 \rangle_{0,3} = 1$
and $\langle \tau_1 \rangle_{1,1} = \frac{1}{24}$, see for example~\cite{ItzZub92}.

For our purpose, it is convenient to consider the \emph{Kontsevich polynomial} $V_{g,n}\in\QQ[x_1,\dots,x_n]$ that gathers the intersection number into a symmetric polynomial on $n$ variables. More precisely,
\begin{align*}
V_{g,n}(x_1,\dots,x_n) &\coloneqq \frac{1}{2^{3g-3+n}} \sum_{\substack{(d_1,\dots,d_n)\in\ZZ^n_{\geq 0} \\ d_1+\cdots+d_n=3g-3+n}} \frac{\langle\tau_{d_1}\cdots\tau_{d_n}\rangle_{g,n}}{d_1!\cdots d_n!} \cdot x_1^{2d_1}\cdots x_n^{2d_n} \\
&= \int_{\overline{\cM}_{g,n}} \exp \left( \sum_{i=1}^n \frac{x_i^2}{2} \psi_i \right).
\end{align*}
For later use we gather the list of small Kontsevich polynomials below
\begin{align*}
V_{0,3}(x_1,x_2,x_3) &= 1, \\
V_{0,4}(x_1,x_2,x_3,x_4) &= \frac{1}{2} (x_1^2 + x_2^2 + x_3^2 + x_4^2), \\
V_{1,1}(x_1) &= \frac{1}{48} x_1^2, \\
V_{1,2}(x_1,x_2) &= \frac{1}{192} (x_1^2 + x_2^2)^2. \\
\end{align*}

\subsection{Random multicurves}
M.\,Mirzakhani proved the polynomial growth of the number of multicurves on hyperbolic surfaces with respect to its length.
This result and some extensions of it are nicely presented in the book of V.~Erlandsson and J.~Souto~\cite{ES}.

Let $X$ be a hyperbolic surface of genus $g$. We define
\begin{equation} \label{eq:sXRgamma}
s_X(R,\gamma) \coloneqq \# \{\eta \in \Mod(X) \cdot \gamma : \ell_X(\eta) \leq N\}.
\end{equation}
\begin{thm}[{\cite[Theorem 1.1, 1.2 and 5.3]{Mir08}}]\label{thm:s=cN}
    Let $X$ be a hyperbolic surface.
    For any multicurve $\gamma\in\cML_X(\ZZ)$ there exists a positive rational constant $c(\gamma)$ such that we have as $R \to \infty$,
    \[
        |s_X(R,\gamma)|
        \sim B(X) \cdot \frac{c(\gamma)}{b_{g}} \cdot R^{6g-6}
    \]
    where $B(X)$ is the Thurston volume of the unit ball in the space of measured laminations $\cML_X$ with respect to the length function $\ell_X$, and
    \[
        b_g
        =
        \sum_{[\gamma]\in\cML_X(\ZZ)/\Mod(X)} c(\gamma)
        =
        \int_{\cM_g} B(X)\, dX.
    \]
\end{thm}
The above theorem allows to give sense to the notion of a random multicurve. Namely we endow the set of topological types of multicurves $\cML_X(\ZZ) / \Mod(X)$ with the probability measure
which assigns $c(\gamma)/b_g$ to $[\gamma]$. We now provide the explicit expression for this probability. For $\Gamma \in \cG_g$ a stable graph we define the
polynomial $F_\Gamma$ on the variables $\{x_e\}_{e \in E(\Gamma)}$ by
\begin{equation} \label{eq:stableGraphPolynomial}
    F_\Gamma \left(\{x_e\}_{e \in E(\Gamma)} \right) =
    \prod_{e \in E(\Gamma)} x_e \cdot \prod_{v \in V(\Gamma)} V_{g_v,n_v}(\bm{x}_v),
\end{equation}
where $\bm{x}_v$ is the multiset of variables $x_e$ where $e$ is an edge adjacent to $v$ and
$V_{g_v,n_v}$ are the Kontsevich polynomial defined in Section~\ref{subsec:PsiClasses}. In
the case $e$ is a loop based at $v$, the variable $x_e$ is repeated twice in $\bm{x}_v$.

\begin{figure}[!ht]
\begin{tabular}{c|c|l}
\hline
multicurve & stable graph & polynomial $F_\Gamma$ \\
\hline \hline
\raisebox{-.5\height}{\includegraphics[scale=0.3]{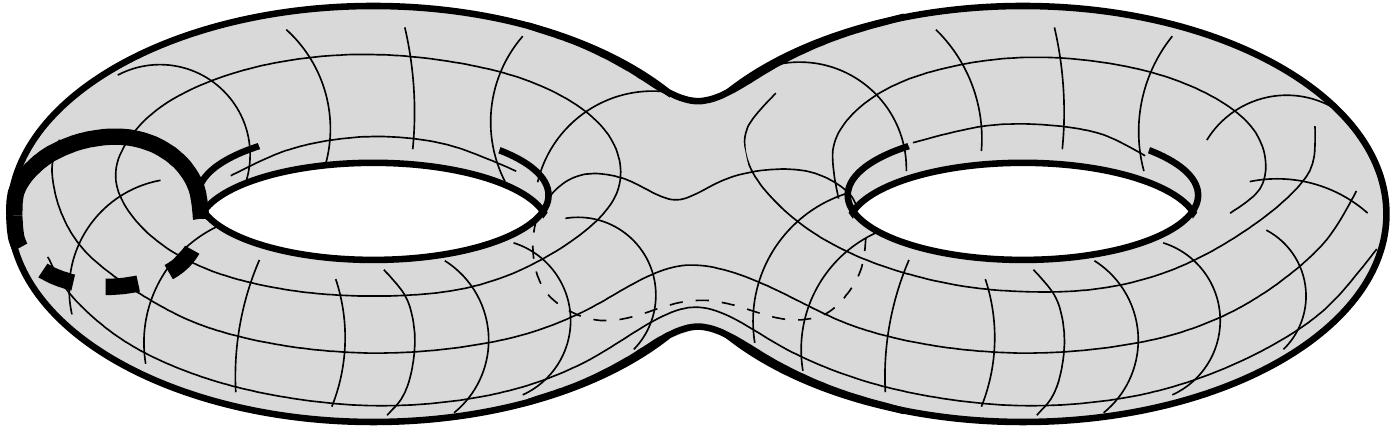}} &
\raisebox{-.5\height}{\includegraphics[scale=1]{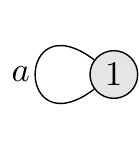}} &
\begin{minipage}{0.4\textwidth}
$x_a V_{1,2}(x_a, x_a)$ \\
$= \frac{1}{48} x_a^5$
\end{minipage}
\\
\hline
\raisebox{-.5\height}{\includegraphics[scale=0.3]{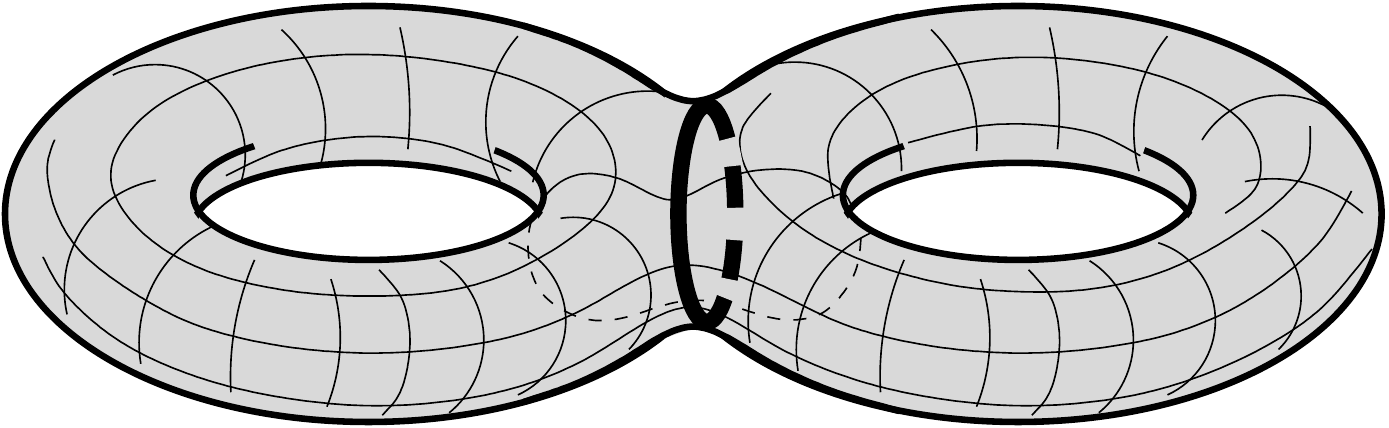}} &
\raisebox{-.5\height}{\includegraphics[scale=1]{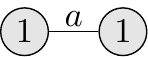}} &
\begin{minipage}{0.4\textwidth}
$x_a V_{1,1}(x_a) V_{1,1}(x_a)$ \\
$= \frac{1}{2304} x_a^5$
\end{minipage}
\\
\hline
\raisebox{-.5\height}{\includegraphics[scale=0.3]{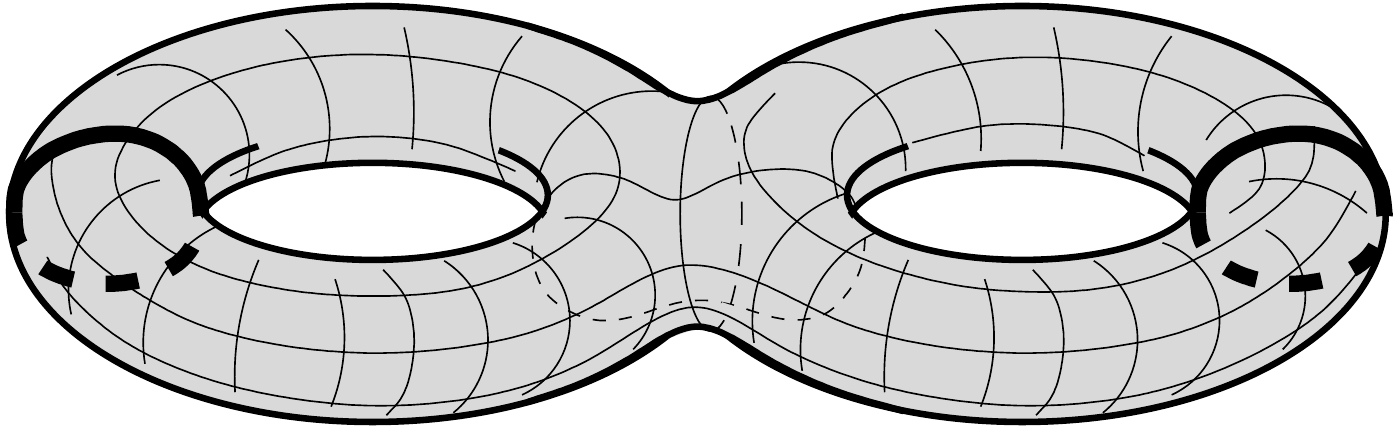}} &
\raisebox{-.5\height}{\includegraphics[scale=1]{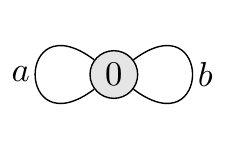}} &
\begin{minipage}{0.4\textwidth}
$x_a x_b V_{0,4}(x_a, x_a, x_b, x_b)$ \\
$= x_a^3 x_b + x_a x_b^3$
\end{minipage}
\\
\hline
\raisebox{-.5\height}{\includegraphics[scale=0.3]{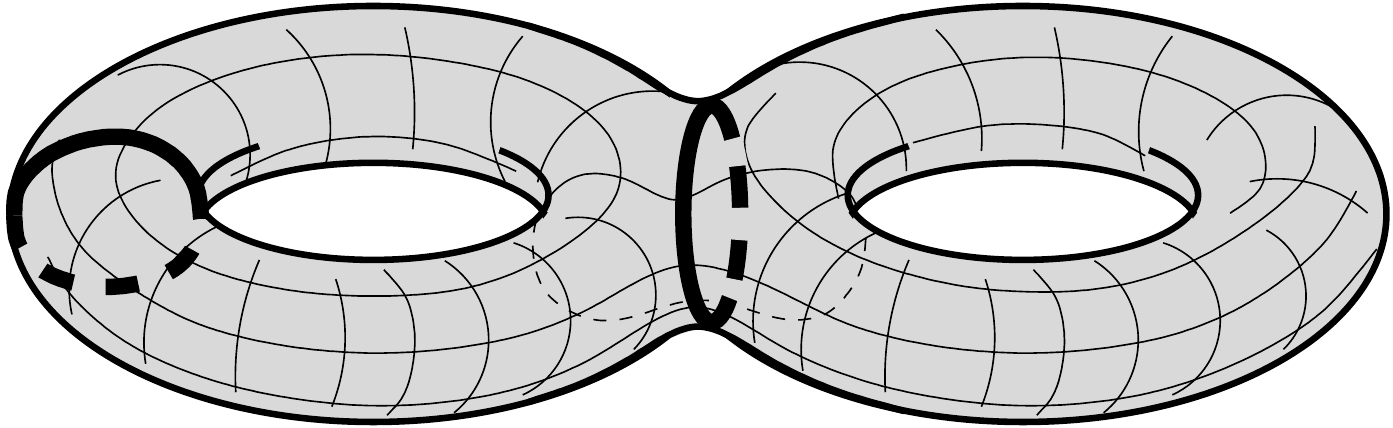}} &
\raisebox{-.5\height}{\includegraphics[scale=1]{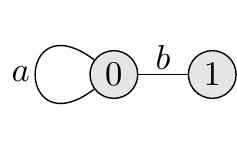}} &
\begin{minipage}{0.4\textwidth}
$x_a x_b V_{0,3}(x_a, x_a, x_b) V_{1,1}(x_b)$ \\
$= \frac{1}{48} x_a x_b^3$
\end{minipage}
\\
\hline
\raisebox{-.5\height}{\includegraphics[scale=0.3]{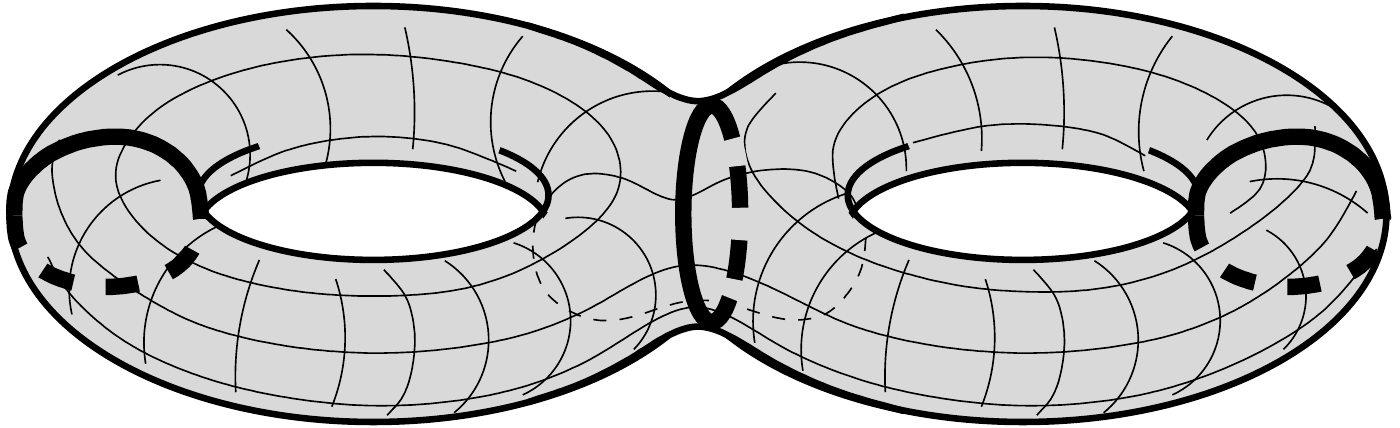}} &
\raisebox{-.5\height}{\includegraphics[scale=1]{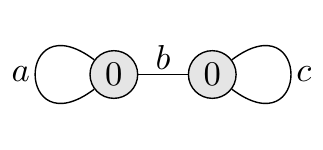}} &
\begin{minipage}{0.4\textwidth}
$x_a x_b x_c V_{0,3}(x_a,x_a,x_b) V_{0,3}(x_b,x_c,x_c)$ \\
$= x_a x_b x_c$
\end{minipage}
\\
\hline
\raisebox{-.5\height}{\includegraphics[scale=0.3]{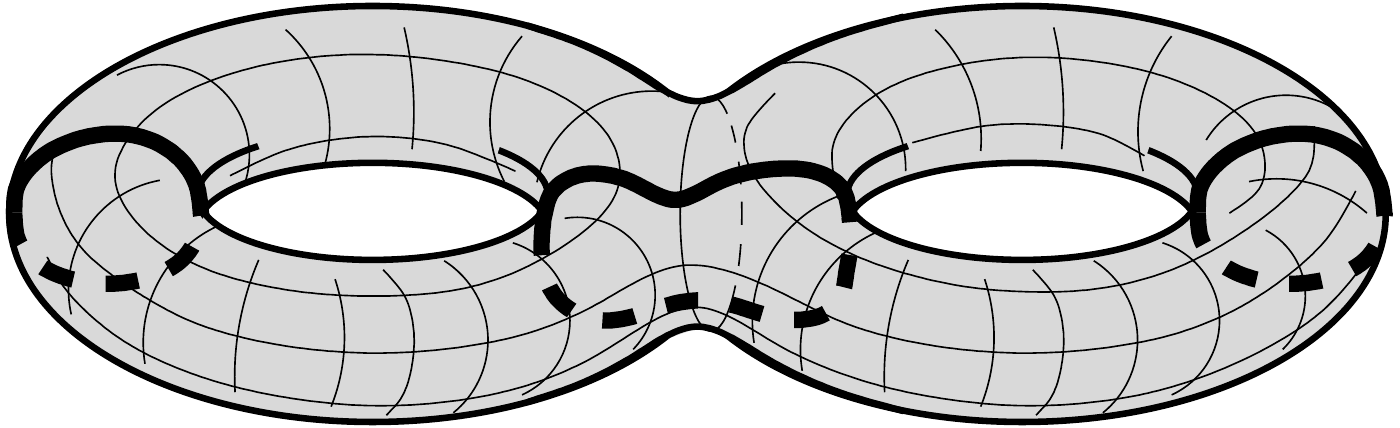}} &
\raisebox{-.5\height}{\includegraphics[scale=1]{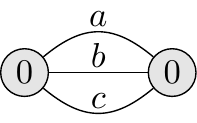}} &
\begin{minipage}{0.4\textwidth}
$x_a x_b x_c V_{0,3}(x_a,x_b,x_c) V_{0,3}(x_a,x_b,x_c)$ \\
$= x_a x_b x_c$
\end{minipage}
\\
\hline
\end{tabular}
\caption{The list of topological types of primitive multicurves in genus 2,
their associated stable graphs and their corresponding polynomial $F_\Gamma$.
The labels on edges are used as variable indices in $F_\Gamma$.}
\end{figure}

\begin{rem} \label{rk:normalization}
The polynomial $F_\Gamma$ appeared first in Mirzakhani's work~\cite{Mir08}, see in particular Theorem 5.3. They were related to square-tiled surfaces and Masur--Veech volumes in~\cite{DGZZ21} though with a different normalization. Namely, the polynomial $P_\Gamma$ from~\cite{DGZZ21} is related to $F_\Gamma$ by
\[
    P_\Gamma = 2^{4g-2} \cdot \frac{(4g-4)!}{(6g-7)!} \cdot \frac{1}{|\Aut (\Gamma) |} \cdot F_\Gamma.
\]
The normalization of $F_\Gamma$ is identical to the conventions used in~\cite{ABCDGLW} and simplifies the computations of the present article.
\end{rem}

Following~\cite{DGZZ21}, for a weighted stable graph $(\Gamma, \boldsymbol{m})$ and we denote by $\cY_{\boldsymbol{m}}\colon \QQ[\{x_e\}_{e \in E(\Gamma)}] \to \QQ$ the linear operator defined on monomials
\[
    \cY_{\boldsymbol{m}}(x_1^{n_1}x_2^{n_2}\cdots x_k^{n_k})
    \coloneqq
    \frac{n_1! n_2!\cdots n_k!}{\boldsymbol{m}_1^{n_1+1}\boldsymbol{m}_2^{n_2+1}\cdots \boldsymbol{m}_k^{n_k+1}}
\]
and for $m \in \NN \cup \{+\infty\}$, set
\[
    \cZ_m \coloneqq \sum_{\substack{\boldsymbol{m} \in \NN^{E(\Gamma)} \\ \forall e \in E(\Gamma), \boldsymbol{m}_e \le m}} \cY_{\boldsymbol{m}}.
\]
We derive the following directly from~\cite{DGZZ21}:
\begin{thm} \label{thm:c=vol}
Let $\gamma$ be a multicurve in genus $g$ and $(\Gamma, \boldsymbol{m})$ the dual weighted stable graph. Then
\[
c(\gamma) = \frac{1}{(6g-6)!}\, \frac{1}{|\Aut(\Gamma, \boldsymbol{m})|} \, \cY_{\boldsymbol{m}}(F_\Gamma).
\]
Furthermore
\[
b_g = \frac{1}{(6g-6)!}\, \sum_{\Gamma \in \cG_g} \frac{1}{|\Aut(\Gamma)|} \, \cZ(F_\Gamma).
\]
\end{thm}

\begin{rem} \label{rk:automorphisms}
In Theorem~\ref{thm:c=vol} we fix a misconception in~\cite{DGZZ21} about automorphisms of multicurves (or equivalently weighted stable graph). Indeed, the way we defined automorphisms of stable graphs and weighted stable graphs in Section~\ref{ssec:multicurvesAndStableGraphs} make it so that the following formula is valid
\[
\sum_{\Gamma} \frac{1}{|\Aut(\Gamma)|} \cZ(F_\Gamma) = \sum_{(\Gamma, \boldsymbol{m})} \frac{1}{|\Aut(\Gamma, \boldsymbol{m})|} \cY_{\boldsymbol{m}}(F_\Gamma)
\]
where the sums are taken over isomorphism classes of respectively stable graphs of genus $g$ and weighted stable graphs of genus $g$.
\end{rem}

\begin{proof}
Up to the correction of Remark~\ref{rk:automorphisms} this is exactly~\cite[Theorem 1.22]{DGZZ21} (see Remark~\ref{rk:normalization} for the difference between $P_\Gamma$ and $F_\Gamma$).
\end{proof}

\begin{figure}[!ht]
\begin{tabular}{c|l}
\hline
stable graph $\Gamma$ & value of $\cY_{\boldsymbol{m}}(\Gamma)$ \\
\hline \hline
\raisebox{-.5\height}{\includegraphics[scale=1]{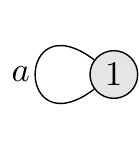}} &
$\displaystyle\frac{5}{2} \cdot \frac{1}{\boldsymbol{m}_a^6}$
\\
\hline
\raisebox{-.35\height}{\includegraphics[scale=1]{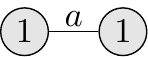}} &
$\displaystyle\frac{5}{96} \cdot \frac{1}{\boldsymbol{m}_a^6}$
\\
\hline
\raisebox{-.5\height}{\includegraphics[scale=1]{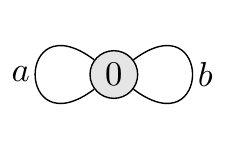}} &
$\displaystyle 24 \left( \frac{1}{\boldsymbol{m}_a^4 \cdot \boldsymbol{m}_b^2} + \frac{1}{\boldsymbol{m}_a^2 \cdot \boldsymbol{m}_b^4} \right)$
\\
\hline
\raisebox{-.5\height}{\includegraphics[scale=1]{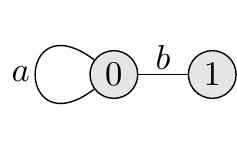}} &
$\displaystyle \frac{1}{2} \cdot \frac{1}{\boldsymbol{m}_a^2 \cdot \boldsymbol{m}_b^4}$
\\
\hline
\raisebox{-.5\height}{\includegraphics[scale=1]{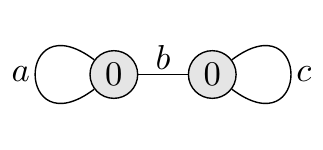}} &
$\displaystyle \frac{1}{\boldsymbol{m}_a^2 \cdot \boldsymbol{m}_b^2 \cdot \boldsymbol{m}_c^2}$
\\
\hline
\raisebox{-.5\height}{\includegraphics[scale=1]{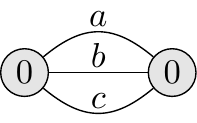}} &
$\displaystyle \frac{1}{\boldsymbol{m}_a^2 \cdot \boldsymbol{m}_b^2 \cdot \boldsymbol{m}_c^2}$
\\
\hline
\end{tabular}
\caption{The list of topological types of primitive multicurves in genus 2
and the associated values $\cY_{\boldsymbol{m}}(\Gamma)$ that is proportional
to $c(\gamma)$ (see Theorem~\ref{thm:c=vol}).}
\end{figure}

\subsection{Asymptotics of $\psi$-correlators and $b_g$}
Our proof of Theorem~\ref{thm:PD} uses crucially the asymptotics of $\psi$-intersections and Masur--Veech volumes from~\cite{Agg21} and further developed in~\cite{DGZZ20}.

\begin{thm}[\cite{Agg21}]  \label{thm:AggarwalAsymptotics}
For $g,n\in\NN$ and $\bm{d} = (d_1, \ldots, d_n)\in\ZZ_{\geq 0}^n$ with $d_1 + \cdots + d_n = 3g-3+n$, let $\epsilon(\bm{d})$ be defined by
\[
    \langle\tau_{d_1} \cdots \tau_{d_n}\rangle_{g,n} = \frac{(6g-5+2n)!!}{(2d_1+1)!!\cdots (2d_n+1)!!} \cdot \frac{1}{g! \cdot 24^g} \cdot (1+\epsilon(\bm{d})).
\]
Then
\[
\lim_{g \to \infty}
\sup_{n < \sqrt{g/800}}
\sup_{\substack{\bm{d} = (d_1, \ldots, d_n) \\ d_1 + \cdots + d_n = 3g-3+n}} \epsilon(\bm{d}) = 0.
\]
\end{thm}

For $m \in \NN \cup \{+\infty\}$ we define
\begin{equation} \label{eq:bgm}
b_{g,m} := \frac{1}{(6g-6)!} \sum_{\Gamma \in \cG_g} \frac{\cZ_m(\Gamma)}{|\Aut(\Gamma)|}.
\end{equation}
Note that $b_g = b_{g,+\infty}$.
\begin{rem}
We warn the reader that the constant denoted $b_{g,m}$ in this article has
nothing to do with the analogue of $b_g$ in the context of surfaces of
genus $g$ with $n$ boundaries which is denoted $b_{g,n}$ in~\cite{Mir16}
and~\cite{DGZZ21}.
\end{rem}
For $m \in \NN \cup \{+\infty\}$ and a real number $\kappa > 1$ we also define
\begin{equation} \label{eq:bgmkappa}
\tilde{b}_{g,m,\kappa} \coloneqq
\frac{1}{(6g-6)!}
\sum_{\substack{\Gamma \in \cG_g \\ |V(\Gamma)| = 1 \\ |E(\Gamma)| \le \kappa \frac{\log(6g-6)}{2}}} \frac{1}{|\Aut(\Gamma)|}
\sum_{\substack{\boldsymbol{m} \in \NN^{E(\Gamma)} \\ \forall e \in E(\Gamma), \boldsymbol{m}_e \leq m}} \cY_{\boldsymbol{m}}(F_\Gamma).
\end{equation}
As we have less terms in its definition, $\tilde{b}_{g,m,\kappa} \le b_{g,m}$.

We will use the asymptotic results of~\cite{Agg21} and~\cite{DGZZ20} in the following form.
\begin{thm}[\cite{Agg21}, \cite{DGZZ20}] \label{thm:A+DGZZtruncatedSum} \label{thm:A+DGZZbgAsymptotics}
Let $m \in \NN \cup \{+\infty\}$ and $\kappa > 1$. Then as $g \to \infty$ we have
\[
b_{g,m} \sim \tilde{b}_{g,m,\kappa} \sim \frac{1}{\pi} \cdot \frac{1}{(6g-6) \cdot (4g-4)!} \cdot \sqrt{\frac{m}{m+1}} \cdot \left( \frac{4}{3} \right)^{4g-4}.
\]
\end{thm}

%%%%%%%%%%%%%%%%%%%%%%%%%%%%%%%%%%%%%%%%%%%%%%%%%%%%%%%%%%%%%%%%%%%%%%%%%%%%%%%%%
\section{Length vectors of random multicurves}
\label{sec:RandomMulticurves}
The aim of this section is to state and prove a refinement of Theorem~\ref{thm:Lg} that provides an explicit description of the random variable $L^{(g,m)}$. For each weighted stable graph $(\Gamma, \boldsymbol{m})$ we define a random variable $U^{(\Gamma, \boldsymbol{m})}$. We then explain how $L^{(g,m)}$ is obtained from them.

Let $\Gamma$ be a stable graph and let $k \geq |E(\Gamma)|$. For
each injection $\iota \colon E(\Gamma) \to \{1,2,\ldots,k\}$,
we define an injection
$g_{\Gamma,\iota} \colon \RR^{E(\Gamma)} \to \RR^k$
by
\[
    g_{\Gamma,\iota}(\{x_e\}_{e \in E(\Gamma)}) = (y_1, y_2, \ldots, y_k),
    \qquad
    \text{where }
    y_i = \left\{ \begin{array}{ll}
            x_{\iota^{-1}(i)} & \text{if $i \in \iota(E(\Gamma))$} \\
            0 & \text{otherwise.}
    \end{array} \right.
\]
Given a measure $\mu$ on $\RR^{E(\Gamma)}$ we define its \emph{$k$-th symmetrization}
to be the measure on $\RR^k$ given by
\[
G_{\Gamma,k}(\mu) \coloneqq \frac{(k - |E(\Gamma)|)!}{k!} \sum_{\iota} (g_{\Gamma,\iota})_* (\mu).
\]
The $k$-th symmetrization is supported on the subspaces of dimension $|E(\Gamma)|$
generated by basis vectors. Because of the coefficient $(k - |E(\Gamma)|)!)/k!$, the
total weights of the measures $\mu$ and $(g_{\Gamma,\iota})_* \mu$ are the same.

We prove the following refinement of Theorem~\ref{thm:Lg}.
\begin{thm}
\label{thm:LgMorePrecise}
Let $L^{(g,m)}$ be the random variable on $\Delta^{3g-3}_{\leq 1}$ with density
\begin{equation} \label{eq:LgAsSum}
\frac{1}{(6g-6) \cdot b_{g,m}}
\sum_{\Gamma \in \cG_g} \frac{1}{|\Aut(\Gamma)|} \sum_{\substack{\boldsymbol{m} \in \NN^{E(\Gamma)} \\ \forall e \in E(\Gamma), \boldsymbol{m}_e \le m}} G_{\Gamma,3g-3}(\mu_{\Gamma,\boldsymbol{m}})
\end{equation}
where $b_{g,m}$ is defined in~\eqref{eq:bgm} and $\mu_{\Gamma,m}$ is the measure on $\Delta^{E(\Gamma)}_{=1}$ with density
\begin{equation} \label{eq:muGammam}
F_\Gamma\left(\left\{\frac{x_e}{\boldsymbol{m}_e}\right\}_{e \in E(\Gamma)}\right)
\cdot
\prod_{e \in E(\Gamma)} \frac{1}{\boldsymbol{m}_e}
\end{equation}
where $F_\Gamma$ is the polynomial defined in~\eqref{eq:stableGraphPolynomial}.
Then
\begin{equation} \label{eq:SumAllCurves}
    \frac{1}{s_X(R, m)} \sum_{\substack{\gamma \in \cML_X(\ZZ) \\ \ell_X(\gamma) \leq R \\ \multiplicity{\gamma} \le m}} \delta_{\hat{\Ell}_X^\da(\gamma)} \xrightarrow[R \to \infty]{} L^{(g,m)\downarrow}
\end{equation}
where $s_X(R, m) \coloneqq \# \{\gamma \in \cML_X(\ZZ) : \ell_X(\gamma) \leq R \text{ and } \multiplicity(\gamma) \leq m\}$ is the number of multicurves on $R$ of length at most $R$ and multiplicity at most $m$ and $L^{(g,m)\downarrow}$ is the vector $L^{(g,m)}$ sorted in decreasing order.
\end{thm}

The study of the length vector of multicurves of a given topological type was initiated by M.\,Mirzakhani in~\cite{Mir16}. She studied the special case of maximal multicurve corresponding to a pants decomposition. The general case that we present now was proved independently in~\cite{AH21} and~\cite{Liu19}.
\begin{thm}[\cite{AH21}, \cite{Liu19}] \label{thm:AHLiu}
Let $X$ be a hyperbolic surface and $\gamma$ a multicurve on $X$ with $k$ components. Let $(\Gamma, \boldsymbol{m})$ be a weighted stable graph dual to $\gamma$. Let $U^{(\Gamma,\boldsymbol{m})}$ be the random variable on $\Delta^{k}_{=1}$ with density
    \begin{equation}\label{eq:pdf}
\frac{(6g-7)!}{\cY_{\boldsymbol{m}}(F_\Gamma)} \frac{1}{\boldsymbol{m}_1 \cdots \boldsymbol{m}_k} \cdot F_\Gamma \left(\frac{x_1}{\boldsymbol{m}_1}, \dots, \frac{x_k}{\boldsymbol{m}_k} \right)
    \end{equation}
    where $E(\Gamma)$ and $V(\Gamma)$ are the set of edges and the set of vertices of $\Gamma$, respectively. Then we have the convergence in distribution
\[
\frac{1}{s_X(R,\gamma)}
\sum_{\substack{\eta \in \Mod(X) \cdot \gamma \\ \ell_X(\eta) \leq R}} \delta_{\hat{\Ell}_X^\da(\eta)} \xrightarrow[R \to \infty]{} U^{(\Gamma,\boldsymbol{m})\downarrow}
\]
where $U^{(\Gamma,\boldsymbol{m})\downarrow}$ is the sorted version of $U^{(\Gamma,\boldsymbol{m})}$ and $s_X(R,\gamma)$ is defined in~\eqref{eq:sXRgamma}.
\end{thm}

We endow $\Delta^k_{\le r}$ with the restriction of the Lebesgue measure on $\RR^k$ that we denote by $\lambda_{\le r}^k$. We define the slice $\Delta_{=r}^k$ inside $\Delta^k_{\le r}$ as
\[
\Delta_{=r}^k
\coloneqq
\{(x_1,x_2,\ldots,x_k) \in [0,\infty)^k : x_1 + x_2 + \cdots + x_k = r\}
\]
and its infinite counterpart
\[
\Delta_{=r}^\infty
\coloneqq
\{(x_1,x_2,\ldots) \in [0,\infty)^\NN : x_1 + x_2 + \cdots = r\}.
\]
Let us mention that $\Delta_{\le r}^\infty$ is closed for the product topology in $[0,r]^\NN$ and hence compact. However $\Delta_{=r}^\infty$ is dense in $\Delta_{\le r}^\infty$. For this reason, it is more convenient to work with measures on $\Delta_{\le r}^\infty$ even though they are ultimately supported on $\Delta_{=r}^\infty$.

On $\Delta^k_{= r}$ which is contained in a hyperplane in $\RR^k$ we consider the Lebesgue measure induced by any choice of $k-1$ coordinates among $x_1$, \ldots, $x_k$. The latter measure is well defined since the change of variables between different choices has determinant $\pm 1$. We first start with an elementary integration lemma.
\begin{lem}\label{lem:int_sp}
Let $d_1,\dots,d_k\in\RR_{\geq 0}$. Then
\[
    \int_{\Delta_{\leq r}^k} x_1^{d_1}x_2^{d_2}\cdots x_k^{d_k} \, d\lambda^k_{\leq r}
    = \frac{d_1!\cdots d_k!}{(d_1+\cdots+d_k+k)!} \cdot r^{d_1+\cdots+d_k+k}
\]
and
\[
    \int_{\Delta_{=r}^k} x_1^{d_1}x_2^{d_2}\cdots x_k^{d_k} \, d\lambda^k_{=r}
    = \frac{d_1!\cdots d_k!}{(d_1+\cdots+d_k+k-1)!} \cdot r^{d_1+\cdots+d_k+k-1},
\]
\end{lem}
Here the factorial of a real number has to be considered by mean of the analytic continuation
given by the gamma function : $x! = \varGamma(x+1)$.

\begin{rem} \label{rk:proba}
Using Lemma~\ref{lem:int_sp}, let us check that~\eqref{eq:LgAsSum} and~\eqref{eq:pdf} are indeed densities of probability measures. From the second equation in the statement of Lemma~\ref{lem:int_sp} it follows that the total mass of~\eqref{eq:muGammam} is $\cY_m(F_\Gamma)$, namely
\[
\cY_m(F_\Gamma) = (6g-7)! \int_{\Delta^k_{=1}} \frac{1}{\boldsymbol{m}_1\cdots \boldsymbol{m}_k} \, F_\Gamma \left( \frac{x_1}{\boldsymbol{m}_1}, \ldots, \frac{x_k}{\boldsymbol{m}_k}\right) d\lambda^k_{=1}(x_1, \ldots, x_k).
\]
Indeed, each monomial that appears in $F_\Gamma$ has $k$ variables and total degree $6g-6-k$. Hence the denominator coming from the formula of Lemma~\ref{lem:int_sp} compensates the $(6g-7)!$ term from~\eqref{eq:muGammam}. The numerator in the formula of Lemma~\ref{lem:int_sp} matches the definition of $\cY_m$.
\end{rem}

\begin{proof}[Proof of Lemma~\ref{lem:int_sp}]
For $x > 0$ real and $\alpha,\beta \geq 1$ integral, we have the
following scaling of the beta function
\begin{equation}\label{eq:beta}
    \int_0^x t^{\alpha-1}(x-t)^{\beta-1}\, dt
    =
    \frac{\varGamma(\alpha) \, \varGamma(\beta)}{\varGamma(\alpha+\beta)}\, x^{\alpha+\beta-1}.
\end{equation}
This implies that
\[
\int_{\Delta^k_{\leq r}} \frac{x_1^{d_1}}{d_1!} \frac{x_2^{d_2}}{d_2!} \cdots \frac{x_k^{d_k}}{d_k!} \, d\lambda^k_{\leq r}
=
\int_{\Delta^{k-1}_{\leq r}}
\frac{x_1^{d_1}}{d_1!} \frac{x_2^{d_2}}{d_2!} \cdots \frac{x_{k-1}^{d_{k-1}+d_k+1}}{(d_{k-1}+d_k+1)!} \, d\lambda^{k-1}_{\leq r}
\]
and
\[
\int_{\Delta^k_{=r}} \frac{x_1^{d_1}}{d_1!} \frac{x_2^{d_2}}{d_2!} \cdots \frac{x_k^{d_k}}{d_k!} \, d\lambda^k_{=r}
=
\int_{\Delta^{k-1}_{=r}}
\frac{x_1^{d_1}}{d_1!} \frac{x_2^{d_2}}{d_2!} \cdots \frac{x_{k-1}^{d_{k-1}+d_k+1}}{(d_{k-1}+d_k+1)!} \, d\lambda^{k-1}_{=r}.
\]
The two equations in the statement then follow by induction.
\end{proof}

\begin{proof}[Proof of Theorem~\ref{thm:LgMorePrecise}]
We just have to gather the different contributions of each multicurve coming from Theorem~\ref{thm:AHLiu} that F.~Arana-Herrera and M.~Liu gave. From Theorem~\ref{thm:s=cN} of M.~Mirzakhani, for any multicurve $\gamma\in\cML_X(\ZZ)$, its asymptotic density in $\cML_X(\ZZ)$ is $\frac{c(\gamma)}{b_g}$. Now Theorem~\ref{thm:c=vol} provides the values of $c(\gamma)$ and $b_{g,m}$ in terms of the stable graph polynomials $F_\Gamma$.
\end{proof}

\begin{rem}
    In~\cite{Liu19}, the length vector $\Ell_X(\gamma)$ for a multicurve $\gamma = \boldsymbol{m}_1\gamma_1 + \cdots + \boldsymbol{m}_k\gamma_k$ is defined as $(\ell_X(\gamma_1),\dots,\ell_X(\gamma_k))$ instead of $(\ell_X(\boldsymbol{m}_1 \gamma_1),\dots,\ell_X(\boldsymbol{m}_k\gamma_k))$ as in this paper, and the limit distribution is thus supported on the (non-standard) simplex
    \[
        \{(x_1,\dots,x_k)\in\RR^k : x_1,\dots,x_k\geq 0,\ \boldsymbol{m}_1x_1 + \cdots + \boldsymbol{m}_kx_k = 1\}.
    \]
    As a consequence, our distribution is the push-forward of the distribution in~\cite{Liu19} under the map $(x_1,\dots,x_k) \mapsto (\boldsymbol{m}_1x_1,\dots,\boldsymbol{m}_kx_k)$.
\end{rem}

\section{Reduction in the asymptotic regime}
\label{sec:reduction}
The random variable $L^{(g,m)}$ appearing in Theorem~\ref{thm:LgMorePrecise} is delicate
to study because it involves a huge number of terms. Using Theorem~\ref{thm:A+DGZZtruncatedSum}
from~\cite{Agg21} and~\cite{DGZZ20} we show that we can restrict to a sum involving only
$O(\log(g))$ terms associated to non-separating multicurves.

We denote by $\Gamma_{g,k}$ the stable graph of genus $g$
with a vertex of genus $g-k$ and $k$ loops. To simplify the notation we fix a bijection
between the edges of $\Gamma_{g,k}$ and $\{1,2,\ldots,k\}$ so that $F_{\Gamma_{g,k}}$
is a polynomial in $\QQ[x_1, \ldots, x_k]$. Note that because the edges in $\Gamma_{g,k}$
are not distinguishable, the polynomial $F_{\Gamma_{g,k}}$ is symmetric.

Using the same notation as in Theorem~\ref{thm:LgMorePrecise} we have the following
result.
\begin{thm} \label{thm:reduction}
For $m \in \NN \cup \{+\infty\}$ and $\kappa > 1$, let $\tilde{L}^{(g,m,\kappa)}$ be the random variable on $\Delta^{3g-3}_{\leq 1}$ with density
\begin{equation} \label{eq:LgkappaAsSum}
    \frac{1}{(6g-6) \cdot \tilde{b}_{g,m,\kappa}}
    \sum_{k=1}^{\kappa \cdot \frac{\log(6g-6)}{2}} \frac{1}{|\Aut \Gamma_{g,k}|} \sum_{\substack{\boldsymbol{m} \in \NN^k \\ \forall i \in \{1, \ldots, k\}, \boldsymbol{m}_i \le m}} G_{\Gamma_{g,k},3g-3}(\mu_{\Gamma_{g,k},\boldsymbol{m}})
\end{equation}
where $\tilde{b}_{g,m,\kappa}$ is defined in~\eqref{eq:bgmkappa}.
Then for any function $h \in L^\infty(\Delta^\infty_{\le 1})$ we have
\[
\EE(h(L^{(g,m)}))
\sim
\EE(h(\tilde{L}^{(g,m,\kappa)}))
\]
as $g\to\infty$.
\end{thm}
Note that terms appearing in the sum~\eqref{eq:LgkappaAsSum} in Theorem~\ref{thm:reduction} form a subset of the terms in the sum~\eqref{eq:LgAsSum} in Theorem~\ref{thm:LgMorePrecise}.

\begin{proof}
By Theorem~\ref{thm:A+DGZZtruncatedSum}, a random multicurve of high genus is almost surely non-separating with less than $\kappa \frac{\log(6g-6)}{2}$ edges. As $h$ is bounded, we obtain the result.
\end{proof}

%%%%%%%%%%%%%%%%%%%%%%%%%%%%%%%%%%%%%%%%%%%%%%%%%%%%%%%%%%%%%%%%%%%%%%%%%%%%%%%%%
\section{Size-biased sampling and Poisson--Dirichlet distribution}
\label{sec:PDandGEM}

\subsection{Size-biased reodering}
The components of a multicurve are not ordered in any natural way. In Theorem~\ref{thm:Lg} we solve this issue by defining a symmetric random
variable $L^{(g,m)}$ on $\Delta^{3g-3}_{\le 1}$ and making the convergence happen
towards  $L^{(g,m)\downarrow}$ whose entries are sorted in decreasing order.
In this section we introduce another natural way of ordering the entries: the size-biased ordering. Contrarily to the symmetrization or
the decreasing order, it is a random ordering. The size-biased ordering
turns out to be convenient in the proof of Theorem~\ref{thm:Lg}.

We work with vectors $x = (x_1, \ldots, x_k)$ in $\Delta^k_{\leq 1}$. A
\emph{reordering}
of $x$ is a random variable of the form $(x_{\sigma(1)}, \ldots, x_{\sigma(k)})$ where
$\sigma$ is a random permutation in $S_k$. We aim to define the size-biased
reordering $x^* = (x_{\sigma(1)}, x_{\sigma(2)}, \ldots, x_{\sigma(k)})$ of $x$.

The idea under the size-biased reordering is to pick components according to
their values. One can define the random permutation $\sigma$ inductively as
follows. If $x$ is the zero vector, then $x^* = x_\sigma$ where $\sigma$ is
taken uniformly at random. Otherwise, we set $\sigma(1)$ according to
\[
\PP_x(\sigma(1) = i) \coloneqq \frac{x_i}{x_1 + \cdots + x_k}
\]
and define a new vector $y = (x_1, \ldots, \widehat{x}_{\sigma(1)}, \ldots, x_k)$
on $\Delta^{k-1}_{\leq 1}$ which is the vector $x$ with the component $x_{\sigma(1)}$
removed. In order to keep track of the components we denote $\phi: \{1, 2, \ldots, k-1\} \to \{1, 2, \ldots, k\}$
the unique increasing injection such that its image avoids $\sigma(1)$. In other words
\[
\phi(i) \coloneqq
\left\{ \begin{array}{ll}
i & \text{if $1 \le i < \sigma(1)$} \\
i+1 & \text{if $\sigma(1) \le i \le k-1$.}
\end{array} \right.
\]
Assuming that by induction $y$ has a size-biased reordering $\sigma_y$ we
define for $i \in \{1, 2, \ldots, k-1\}$ the other values by $\sigma(\phi(i)) \coloneqq \phi(\sigma_y(i))$.
This defines inductively the size-biased reordering.

A more direct definition can be given as follows. Given $r$ such that at least $r$
components of $x$ are positive, for $1\leq i_1,\dots,i_r \leq k$ distinct integers, we have
\begin{equation}
\label{eq:sizeBiasedPermutation}
\begin{split}
\PP_x(\sigma(1) = i_1, \ldots, \sigma(r) &= i_r)
= \\
& \frac{x_{i_1} x_{i_2} \ldots x_{i_{r}}}{s(s-x_{i_1})(s-x_{i_1}-x_{i_2}) \cdots (s - x_{i_1} - \cdots - x_{i_{r-1}})}
\end{split}
\end{equation}
where $s = x_1 + x_2 + \cdots + x_k$. Note that for $r=k$ we have $x_r = 1 - x_1 - \cdots - x_{r-1}$ and one can
perform a simplification of the last terms in the numerator and denominator.

Now let $X \colon \varOmega \to \Delta^k_{\le 1}$ be a random variable. In
order to define its size-biased reordering $X^* \colon \varOmega \to \Delta^k_{\le 1}$,
we consider for each $x \in \Delta^k_{\le 1}$ independent random variables $\sigma_x$ distributed according to $\PP_x$
as defined above which are furthermore independent from $X$. We then define for each
$\omega \in \varOmega$
\[
X^*(\omega) \coloneqq \sigma_{X(\omega)} \cdot X(\omega)
\]
where $\sigma \cdot x = (x_{\sigma(1)}, \ldots, x_{\sigma(k)})$.

\begin{lem} \label{lem:MarginalSizeBiased}
Let $X$ a random variable on $\Delta^k_{= 1}$ with
density $f_X \colon \Delta^k_{= 1} \to \RR$.
Let $1 \leq r \leq k$. Then the $r$-th marginal of the size-biased reordering of $X$, that is to say
the density of the vector $(X^*_1, \ldots, X^*_r)$ is
\begin{align*}
    & f_{(X_1^*,\dots,X_r^*)}(x_1,\dots,x_r)
    =
    \frac{1}{(k-r)!} \frac{x_1 \cdots x_r}{(1-x_1)\cdots (1-x_1-\cdots -x_{r-1})} \cdot \\
    & \qquad
    \cdot \int_{\Delta_{= 1-x_1-\cdots -x_r}^{k-r}} \sum_{\sigma\in S_{k}} f_X(x_{\sigma(1)},\dots,x_{\sigma(k)}) \, d\lambda^{k-r}_{=1-x_1-\cdots-x_r}(x_{r+1}, \dots, x_{k}),
\end{align*}
\end{lem}

\begin{proof}
    Let us define $g(x_1, \ldots, x_k) \coloneqq \sum_{\sigma \in S_k} f(x_{\sigma(1)}, \ldots, x_{\sigma(k)})$.

We first consider the case $r=k$.
Since $X$ admits a density, almost surely all components are positive and distinct. Hence one can use~\eqref{eq:sizeBiasedPermutation}
to write its density as
\[
f_{X^*}(x_1, \ldots, x_k) = \frac{x_1 \cdots x_{k}}{(1-x_1) \cdots (1-x_1-\cdots-x_{k-1})} \, g(x_1,\ldots,x_k).
\]
In the above formula we used the fact that the sum of $X$ is $s=1$ almost surely.

Now, for $1 \leq r \leq k-1$, the $r$-th marginal is obtained by integrating the free variables
\begin{equation}
\label{eq:marginalIntegral}
f_{(X^*_1, \ldots, X^*_r)}(x_1, \ldots, x_r) = \int_{\Delta^{k-r}_{=1-s}} f_{X^*}(x_1,\ldots,x_k) \, d\lambda^{k-r}_{=1-s}(x_{r+1},\ldots,x_k)
\end{equation}
where $s = x_1 + \cdots + x_r$. For a permutation $\tau \in S_{k-r}$ we define the subsimplex
\[
\Delta_{=1-s;\tau}^{k-r}
\coloneqq
\{(x_1, \ldots, x_{k-r}) \in \Delta_{=1-s}^{k-r} : x_{\tau(1)} > x_{\tau(2)} > \cdots > x_{\tau(k-r)}\}.
\]
We can decompose the integral~\eqref{eq:marginalIntegral} as a sum over these subsimplices
\begin{align*}
    f_{(X^*_1, \ldots, X^*_r)} & (x_1, \ldots, x_r)
=
\frac{x_1 \cdots x_r}{(1-x_1) \cdots (1-x_1 - \cdots - x_{r-1})} \\
& \cdot \sum_{\tau \in S_{r-k}}
\int_{\Delta_{=1-s;\tau}^{k-r}}
\frac{x_{r+1} \cdots x_k\, g(x_1, \ldots, x_k) \, d\lambda^{k-r}_{=1-s}(x_{r+1}, \ldots, x_k)}{(1-s-x_{k+1}) \cdots (1-s-x_{k+1}-\cdots-x_{r-1})}.
\end{align*}
Using the fact that $g$ is symmetric, we can rewrite it by mean of a change of variables on
the standard simplex $\Delta_{=1-s; \id}^{k-r}$
\begin{align*}
& f_{(X^*_1, \ldots, X^*_r)} (x_1, \ldots, x_r)
=
\frac{x_1 \cdots x_r}{(1-x_1) \cdots (1-x_1 - \cdots - x_{r-1})} \cdot \\
&
\int_{\Delta_{=1-s;\id}^{k-r}}
\sum_{\tau \in S(\{r+1,\ldots,k\})}
\frac{x_{\tau(r+1)} \cdots x_{\tau(k)} \, g(x_1, \ldots, x_k) \, d\lambda^{k-r}_{=1-s}(x_{r+1}, \ldots, x_k)}{(1-s-x_{\tau(k+1)}) \cdots (1-s-x_{\tau(k+1)}-\cdots-x_{\tau(r-1)})}.
\end{align*}
Using the facts that
\[
\sum_{\tau \in S(\{r+1,\ldots,k\})}
\frac{x_{\tau(r+1)} \cdots x_{\tau(k)}}{(1-s-x_{\tau(k+1)}) \cdots (1-s-x_{\tau(k+1)}-\cdots-x_{\tau(r-1)})} = 1
\]
and
\begin{align*}
\int_{\Delta_{=1-s;\id}^{k-r}}
\sum_{\sigma \in S_k} & f(x_{\sigma(1)}, \ldots, x_{\sigma(k)}) \, d\lambda^{k-r}_{=1-s}(x_{r+1}, \ldots, x_k)
\\
& =
\frac{1}{(k-r)!}
\int_{\Delta^{k-r}_{=1-s}}
\sum_{\sigma \in S_k} f(x_{\sigma(1)}, \ldots, x_{\sigma(k)}) \, d\lambda^{k-r}_{=1-s}(x_{r+1}, \ldots, x_k)
\end{align*}
we obtain the result.
\end{proof}

We finish this section by mentioning that the size-biased reordering extends to infinite vectors, that is elements on $\Delta^{\infty}_{\le 1}$.

\subsection{Poisson--Dirichlet and GEM distributions}
\label{ssec:PoissonDirichletAndGEM}

Recall that the $\GEM(\theta)$ distribution was defined in the introduction via the stick-breaking process. We also defined the $\PD(\theta)$ as the sorted reordering of $\GEM(\theta)$. The Poisson--Dirichlet distribution admits an intrinsic definition in terms of the Poisson process first introduced by Kingman~\cite{Kin75}. We refer to~\cite[Section~4.11]{ABT03} for this definition. Instead we concentrate on the simpler Griffiths-Engen-McCloskey distribution.

In the introduction we passed from $\GEM(\theta)$ to $\PD(\theta)$. The following result formalizes the equivalence between the two distributions.
\begin{thm}[\cite{DJ89}] \label{thm:PDversusGEM}
Let $X = (X_1, X_2, \ldots)$ be a random variable on $\Delta_{= 1}^\infty$. Let $\theta > 0$. Then the sorted reordering $X^\downarrow$ has distribution $\PD(\theta)$ if and only if the size-biased reordering $X^*$ has distribution $\GEM(\theta)$.
\end{thm}
We will use the above result in the following form.
\begin{cor}[\cite{DJ89}] \label{cor:lim=PD<->lim=GEM}
Let $X^{(n)}$ be a sequence of random variables on $\Delta_{= 1}^\infty$. Let $\theta > 0$. Then the sorted sequence $X^{(n)\downarrow}$ converges in distribution to $\PD(\theta)$ if and only if the size-biased sequence $X^{(n)*}$ converges in distribution to $\GEM(\theta)$.
\end{cor}

In order to prove convergence towards $\GEM$ we will need the explicit description of its marginals.
\begin{prop}[\cite{DJ89}] \label{prop:pdfGEM}
Let $X = (X_1,X_2,\dots)$ be a random variable with distribution $\GEM(\theta)$. Then the distribution the $r$-first components $(X_1,X_2,\dots,X_r)$ of $X$ supported on $\Delta^{r}_{\leq 1}$ admit a distribution with density given by
\begin{equation} \label{eq:GEMdensity}
    \frac{\theta^{r} (1-x_1-\cdots -x_r)^{\theta-1}}{(1-x_1)(1-x_1-x_2)\cdots (1-x_1-\cdots -x_{r-1})}.
\end{equation}
\end{prop}

In order to simplify computations, we consider moments of the $\GEM$ distribution that get
rid of the denominator in the density~\eqref{eq:GEMdensity}. Namely, for a random variable
$X = (X_1, X_2, \ldots)$ on $\Delta^\infty_{\leq 1}$ and $p = (p_1, \ldots, p_r)$ a $r$-tuple
of non-negative integers we define
\begin{equation} \label{eq:MpDefinition}
M_p(X) \coloneqq
\EE((1-X_1) \cdots (1-X_1-\cdots -X_{r-1}) \cdot X_1^{p_1}\cdots X_r^{p_r})
\end{equation}
These moments of $\GEM(\theta)$ are as follows.
\begin{lem}\label{lem:GEMMoments}
If $X=(X_1,X_2,\dots) \sim \GEM(\theta)$ and $(p_1, \ldots, p_r)$ is a non-negative integral
vector, then the moment $M_p(X)$ defined in~\eqref{eq:MpDefinition} has the following value
\[
M_p(X) = \frac{\theta^r \cdot (\theta-1)! \cdot p_1!\cdots p_r!}{\varGamma(p_1+ \cdots + p_r + \theta + r)!}.
\]
\end{lem}

\begin{proof}
By Proposition~\ref{prop:pdfGEM} we have
\begin{align*}
M_p(X) &= \int_{\Delta^r_{\leq 1}} \theta^{r} x_1^{p_1} \cdots x_r^{p_r} (1-x_1-\cdots -x_r)^{\theta-1} \, d\lambda^r_{\le 1}\\
&= \theta^r \int_{\Delta^{r+1}_{=1}} x_1^{p_1} \cdots x_r^{p_r} x_{r+1}^{\theta-1} \, d\lambda^{r+1}_{=1}.
\end{align*}
The last term is an instance of Lemma~\ref{lem:int_sp} on the simplex $\Delta^{r+1}_{=1}$.
Replacing the value obtained from the integration lemma gives the result.
\end{proof}

%%%%%%%%%%%%%%%%%%%%%%%%%%%%%%%%%%%%%%%%%%%%%%%%%%%%%%%%%%%%%%%%%%%%%%%%%%%%%%%%%
\section{Proof of the main theorem}
\label{sec:LargeGenus}

The aim of this section is to prove the following result
\begin{thm}\label{thm:GEM}
For $g \ge 2$ integral, $m \in \NN \cup \{+\infty\}$ and $\kappa > 1$ real, let $\tilde{L}^{(g,m,\kappa)*}$ be the size-biased version of the random variable $\tilde{L}^{(g,m,\kappa)}$ from Theorem~\ref{thm:reduction}. Then as $g$ tends to $\infty$, $\tilde{L}^{(g,m,\kappa)*}$ converges in distribution to $\GEM(1/2)$.
\end{thm}

Let us first show how to derive our main Theorem~\ref{thm:PD} from Theorem~\ref{thm:GEM}.
\begin{proof}[Proof of Theorem~\ref{thm:PD}]
By Theorem~\ref{thm:reduction}, the random variables $L^{(g,m,\kappa)*}$ and $L^{(g,m)*}$ have the same limit distribution as $g \to +\infty$. Hence by Theorem~\ref{thm:GEM}, the random variable $L^{(g,m)*}$ converges in distribution towards $\GEM(1/2)$.

Finally Corollary~\ref{cor:lim=PD<->lim=GEM} shows that the convergence
in distribution of $L^{(g,m)*}$ towards $\GEM(1/2)$ is equivalent to
the convergence of $L^{(g,m)\da}$ towards $\PD(1/2)$. This concludes the proof
of Theorem~\ref{thm:PD}.
\end{proof}

\subsection{Moment's method}
Let us recall from Section~\ref{ssec:PoissonDirichletAndGEM} Equation~\eqref{eq:MpDefinition} that we defined some specific moments $M_{(p_1,\ldots,p_r)}(X)$ for a random variable $X$ on $\Delta^\infty_{=1}$. In this section, we show that the convergence of a sequence of random variables $X^{(n)}$ is equivalent to the convergence of all the moments $M_p(X^{(n)})$. This strategy called the \emph{method of moments} is a standard tool in probability, see for example~\cite[Section~30]{Billingsley} for the case of real variables.
\begin{lem}\label{lem:momentMethod}
A sequence of random variables $X^{(n)} = (X^{(n)}_1,X^{(n)}_2,\ldots)$ in $\Delta_{=1}^\infty$ converges in distribution to a random variable $X^{(\infty)}$ in $\Delta_{=1}^\infty$ if and only if for all $p=(p_1,\ldots,p_r)$ vector of non-negative integers we have $\lim_{n \to \infty} M_p(X^{(n)}) = M_p(X^{(\infty)})$.
\end{lem}

\begin{proof}
The infinite-dimensional cube $[0,1]^\NN$ is compact with respect to the product topology by Tychonoff's theorem. The set $\Delta_{\leq 1}^\infty$ is a closed subset of $[0,1]^\NN$, and is therefore compact. The signed measures on $\Delta_{\leq 1}^\infty$ are identified with the dual of the real continuous function $C(\Delta_{\leq 1}^\infty, \RR)$. In particular, we have the convergence of $X^{(n)}$ towards $X^{(\infty)}$ in distribution if and only if for any continuous function $f \in C(\Delta_{\le 1}^\infty, \RR)$ we have the convergence of $\EE(f(X^{(n)}))$ towards $\EE(f(X^{(\infty)}))$.

Now let $S$ be the set of functions in $C(\Delta_{\leq 1}^\infty, \RR)$ of the form
\[
    (1-x_1)(1-x_1-x_2)\cdots (1-x_1-\cdots -x_{r-1}) \cdot x_1^{p_1}\cdots x_r^{p_r},
\]
with $r\geq 0$, $p_1,\dots,p_r \geq 0$. We claim that the span of $S$ (that is finite linear combinations of elements of $S$) is dense in $C(\Delta_{\leq 1}^\infty, \RR)$.

Indeed, $S$ contains $1$ and is stable under multiplication. Therefore, the algebra generated by $S$ is equal to its span.

Now, the set $S$ is a separating subset of $C(\Delta_{\leq 1}^\infty,\RR)$ and density follows from the Stone--Weierstrass theorem.
\end{proof}

We will use the following asymptotic simplification of the moments.
\begin{thm}\label{thm:LgMomentsAsymptotics}
    For $g \geq 2$ integral, $m \in \NN \cup \{+\infty\}$ and $\kappa > 1$ real, let $\tilde{L}^{(g,m,\kappa)*}$ be the size-biased reordering of the random variable $\tilde{L}^{(g,m,\kappa)}$ from Theorem~\ref{thm:reduction}. Let $r\geq 1$ and $p_1,\dots,p_r\in\NN$.
    Then, as $g \to \infty$, the moment $M_p(\tilde{L}^{(g,m,\kappa)*})$ is asymptotically equivalent to
\[
\frac{\sqrt{\frac{m+1}{m}} \cdot \sqrt{\pi}}{2 \cdot (6g-6)^{p_1+\cdots+p_r+r-1/2}}
\sum_{k=r}^{\kappa \frac{\log(6g-6)}{2}} \frac{1}{(k-r)!} \sum_{\substack{(j_1,\dots,j_k) \in \NN^k \\ j_1+\cdots +j_k = 3g-3}}
    \prod_{i=1}^k \frac{\zeta_m(2 j_i)}{2 j_i} \prod_{i=1}^r \frac{(2 j_i + p_i)!}{(2j_i - 1)!}
\]
where
\[
\zeta_m(s) \coloneqq \sum_{n = 1}^m \frac{1}{n^s}
\]
is the partial Riemann zeta function.
\end{thm}

Following~\cite[Equation~(14)]{DGZZ20}, we define
\[
    \begin{split}
    c_{g,k}(d_1,\dots,d_k)
    & \coloneqq \frac{g! \cdot (3g-3+2k)!}{(6g-5+4k)!} \frac{3^g}{2^{3g-6+5k}}
    \cdot (2d_1+2)!\cdots (2d_k+2)! \\
    & \qquad \cdot \sum_{\substack{d_i^-+d_i^+ = d_i \\ d_i^-,d_i^+\geq 0, 1\leq i\leq k}} \frac{\langle\tau_{d_1^-}\tau_{d_1^+} \cdots \tau_{d_k^-}\tau_{d_k^+}\rangle_{g,2k}}{d_1^-! d_1^+! \cdots d_k^- d_k^+!}.
    \end{split}
\]
The above coefficients were introduced because by~\cite[Lemma~3.5]{DGZZ20}, and we have
\[
    \lim_{g \to \infty} \sup_{k \leq \sqrt{g/800}} |c_{g,k}(d_1,\dots,d_k) - 1| = 0.
\]
This asymptotic result is a direct consequence of Theorem~\ref{thm:AggarwalAsymptotics} of A.~Aggarwal that we stated in the introduction. We define $\tilde{c}_{g,k}(j_1,\ldots,j_k) \coloneqq c_{g-k,k}(j_1-1,\dots,j_k-1)$.
\begin{lem} \label{lem:UgkDensity}
For each $k=1,\ldots,3g-3$, let $U^{(g,m,k)}$ be the random variable on $\Delta^k_{=1}$ with density
\begin{equation} \label{eq:UgkDensity}
\frac{(6g-7)!}{\cZ_m(F_{g,k})} \sum_{\substack{\boldsymbol{m} \in \NN^k \\ \forall i \in \{1, \ldots, k\}, \boldsymbol{m}_i \le m}} F_{g,k} \left(\frac{x_1}{\boldsymbol{m}_1}, \dots, \frac{x_k}{\boldsymbol{m}_k} \right) \cdot \frac{1}{\boldsymbol{m}_1 \cdots \boldsymbol{m}_k}
\end{equation}
where we use the notation $F_{g,k}$ for $F_{\Gamma_{g,k}}$. Then for any $p = (p_1, \ldots, p_r) \in \NN^k$ we have
\begin{align*}
    M_p(U^{(g,k)*})
    & =
    \frac{w_{g,k} \cdot k!}{\cZ_m(F_{g,k}) \cdot (k-r)! \cdot (6g-7 + p_1 + \cdots + p_r + r)!} \\
    & \qquad
    \cdot \sum_{\substack{(j_1,\dots,j_k) \in\NN^k \\ j_1 + \dots + j_k = 3g-3}}
        \tilde{c}_{g,k}(j_1,\dots,j_k) \prod_{i=1}^k \frac{\zeta_m(2j_i)}{2j_i}
        \prod_{i=1}^r \frac{(2j_i+p_i)!}{(2j_i-1)!}
\end{align*}
where
\[
w_{g,k} \coloneqq
\frac{(6g-5-2k)! \cdot (6g-7)!}{(g-k)! \cdot (3g-3-k)!} \cdot \frac{2^{3k-3}}{3^{g-k}}.
\]
\end{lem}
Note that Formula~\eqref{eq:UgkDensity} is the density of a probability measure by Remark~\ref{rk:proba}. It is more precisely the density of the asymptotic normalized vector of length of random multicurves restricted to multicurves of the type $\Gamma_{g,k}$.

\begin{proof}
By definition of the stable graph polynomial we have
\begin{align*}
&F_{g,k}(x_1, x_2, \ldots, x_k)
=
x_1 \cdots x_k \cdot V_{g-k,2k}(x_1, x_1, x_2, x_2, \ldots, x_k, x_k)
= \\
&\frac{1}{2^{3g-3-k}}
\sum_{\substack{(d_1^-,d_1^+,\dots,d_k^-,d_k^+)\in\ZZ_{\geq 0}^{2k} \\ d_1^- + d_1^+ + \cdots + d_k^-+d_k^+ = 3g-3-k}} \frac{\langle \tau_{d_1^-}\tau_{d_1^+}\cdots \tau_{d_k^-}\tau_{d_k^+}\rangle_{g-k,2k}}{d_1^-!d_1^+!\cdots d_k^-! d_k^+!} \, x_1^{2(d^+_1+d^-_1)+1}\cdots x_k^{2(d^+_k+d^-_k)+1}.
\end{align*}
Using the coefficients $\tilde{c}_{g,k}$ defined just above the statement of the lemma, we rewrite the polynomial $F_{g,k}$ as
\[
F_{g,k} =
\frac{(6g-5-2k)!}{(g-k)! \cdot (3g-3-k)!} \cdot \frac{2^{3k-3}}{3^{g-k}}
\sum_{\substack{(j_1,\dots,j_k) \in \NN^k \\ j_1 + \dots + j_k = 3g-3}}
\tilde{c}_{g,k}(j_1,\dots,j_k) \prod_{i=1}^k \frac{x_i^{2j_i-1}}{(2j_i)!}.
\]
Hence the density of $U^{(g,m,k)}$ in~\eqref{eq:UgkDensity} can be rewritten as
\[
\frac{w_{g,k}}{\cZ_m(F_{g,k})}
\sum_{\substack{(j_1,\dots,j_k) \in \NN^k \\ j_1 + \dots + j_k = 3g-3}}
\tilde{c}_{g,k}(j_1,\dots,j_k) \prod_{i=1}^k \zeta_m(2j_i) \frac{x_i^{2j_i-1}}{(2j_i)!}.
\]

Now, by Lemma~\ref{lem:MarginalSizeBiased}, the $r$-th marginal of the sized-biased version $U^{(g,m,\kappa)*}$ of $U^{(g,m,\kappa)}$ is
\begin{align*}
    & \frac{w_{g,k} \cdot k!}{\cZ_m(F_{g,k}) \cdot (k-r)!} \cdot \frac{1}{(1-x_1) \cdots (1-x_1-\cdots-x_{r-1})} \\
    & \qquad \cdot \sum_{\substack{(j_1,\dots,j_k)\in\NN^k \\ j_1 + \dots + j_k = 3g-3}} \tilde{c}_{g,k}(j_1,\dots,j_k) \prod_{i=1}^k \frac{\zeta_m(2j_i)}{(2j_i)!} \prod_{i=1}^r x_i^{2j_i} \\
    & \qquad \qquad \cdot \int_{\Delta_{= 1-x_1-\cdots-x_r}^k} x_{r+1}^{2j_{r+1}-1} \cdots x_k^{2j_k-1} \, d\lambda^k_{=1}(x_{r+1},\ldots,x_k)
\end{align*}
In the above, we used the fact that the density of $U^{(g,m,k)}$ is a symmetric function. Hence the sum over all
permutations of $k$ elements only pops out a $k!$ coefficient.
The value of the integral in the above sum follows from Lemma~\ref{lem:int_sp} and is equal to
\[
\frac{(2j_{r+1}-1)! \cdots (2j_k-1)!}{(2j_{r+1} + \cdots 2j_k - 1)!}
 (1 - x_1 - \cdots - x_r)^{2j_{r+1} + \cdots + 2j_k - 1}.
\]
We end up with the following formula for the distribution of the $r$-th marginal of $U^{(g,m,k)*}$
\begin{align*}
    \frac{w_{g,k} \cdot k!}{\cZ_m(F_{g,k}) \cdot (k-r)!}
    \sum_{\substack{(j_1,\dots,j_k)\in\NN^k \\ j_1 + \dots + j_k = 3g-3}}
    & \frac{\tilde{c}_{g,k}(j_1,\dots,j_k)}{(2j_{r+1} + \cdots 2j_k - 1)!} \cdot
    \prod_{i=1}^k \frac{\zeta_m(2j_i)}{2j_i} \prod_{i=1}^r \frac{x_i^{2j_i}}{(2j_i-1)!} \\
    & \qquad \cdot \frac{(1 - x_1 - \cdots - x_r)^{2j_{r+1} + \cdots + 2j_k - 1}}{(1-x_1) \cdots (1-x_1-\cdots-x_{r-1})}.
\end{align*}
From the above formula and the definition of the moment $M_p$ in~\eqref{eq:MpDefinition}, the moment $M_p(U^{(g,m,\kappa)*})$ equals
\begin{align*}
    & \frac{w_{g,k} \cdot k!}{\cZ_m(F_{g,k}) \cdot (k-r)!}
    \sum_{\substack{(j_1,\dots,j_k\in\NN^k) \\ j_1 + \dots + j_k = 3g-3}}
    \frac{\tilde{c}_{g,k}(j_1,\dots,j_k)}{(2j_{r+1} + \cdots 2j_k - 1)!}
    \prod_{i=1}^k \frac{\zeta_m(2j_i)}{2j_i} \prod_{i=1}^r \frac{1}{(2j_i-1)!} \\
    & \qquad \cdot \int_{\Delta^r_{\le 1}}
    x_1^{2j_1+p_1}\cdots x_r^{2j_r+p_r} (1 - x_1 - \cdots - x_r)^{2j_{r+1} + \cdots + 2j_k - 1} \, d\lambda_{\le 1}^r.
\end{align*}
Lemma~\ref{lem:int_sp} gives the value of the above integral
\begin{align*}
    \int_{\Delta^r_{\le 1}}
    & x_1^{2j_1+p_1}\cdots x_r^{2j_r+p_r} (1 - x_1 - \cdots - x_r)^{2j_{r+1} + \cdots + 2j_k - 1} \, d \lambda_{\le 1}^r \\
    & \qquad = \int_{\Delta^{r+1}_{= 1}}
    x_1^{2j_1+p_1}\cdots x_r^{2j_r+p_r} x_{r+1}^{2j_{r+1} + \cdots + 2j_k - 1} \, d \lambda_{= 1}^{r+1} \\
    & \qquad = \frac{(2j_1+p_1)! \cdots (2j_r+p_r)! \cdot (2j_{r+1} + \cdots + 2j_k-1)!}{(6g-6+p_1+\cdots+p_r + r - 1)!}.
\end{align*}
Substituting the above value in our last expression for $M_p(U^{(g,m,\kappa)*})$ gives the announced formula.
\end{proof}

\begin{proof}[Proof of Theorem~\ref{thm:LgMomentsAsymptotics}]
Because the distribution of $\tilde{L}^{(g,m,\kappa)}$ is a weighted sum of distributions, we can perform the computation of the moments for each term in the sum and gather the result in the end.
More precisely, we have
\begin{equation} \label{eq:momentsLgkappaAsSum}
M_p(\tilde{L}^{(g,m,\kappa)*}) =
\frac{1}{(6g-6)! \cdot \tilde{b}_{g,m,\kappa}}
\sum_{k=1}^{\kappa \frac{\log(6g-6)}{2}}
\frac{\cZ_m(F_{g,k})}{2^k \cdot k!} \cdot M_p(U^{(g,m,k)*})
\end{equation}
where $\tilde{b}_{g,m,\kappa}$ was defined in~\eqref{eq:bgmkappa}.

Now substituting the formula for $M_p(U^{(g,m,k)*})$ from Lemma~\ref{lem:UgkDensity} and the asymptotic value of $\tilde{b}_{g,m,\kappa}$ from Theorem~\ref{thm:A+DGZZtruncatedSum} in the sum~\eqref{eq:momentsLgkappaAsSum}, we have as $g \to \infty$ the asymptotic equivalence
\begin{equation} \label{eq:asymp1}
\begin{split}
M_p(L^{(g,m,\kappa)*})
& \sim
\frac{ (4g-4)! \cdot \pi}{ (6g-7)! \cdot (6g-7+p_1+\cdots+p_r+r)!} \cdot \sqrt{\frac{m+1}{m}} \cdot \left( \frac{3}{4} \right)^{4g-4} \\
& \qquad \cdot \sum_{k=1}^{\kappa \frac{\log(6g-6)}{2}}
\bigg(\frac{1}{2^k \cdot (k-r)!} \cdot \frac{(6g-5-2k)! \cdot (6g-7)! \cdot 2^{3k-3}}{(g-k)! \cdot (3g-3-k)! \cdot 3^{g-k}}  \\
& \qquad \qquad \cdot \sum_{\substack{(j_1,\dots,j_k)\in\NN^k \\ j_1 + \dots + j_k = 3g-3}}
\prod_{i=1}^k \frac{\zeta_m(2j_i)}{2j_i}
\prod_{i=1}^r \frac{(2j_i+p_i)!}{(2j_i-1)!} \bigg)
.
\end{split}
\end{equation}
where we have used that $\tilde{c}_{g,k}(j_1,\dots,j_k) \sim 1$ uniformly in $k \in [1,\kappa \log(6g-6)/2]$. On the one hand, by~\cite[Equation~(3.13)]{DGZZ20} (in the proof of Theorem~3.4) we have
\begin{equation} \label{eq:simplify1}
\frac{(4g-4)! \cdot (6g-5-2k)!}{(6g-7)! \cdot (g-k)! \cdot (3g-3-k)!}
\sim
(6g-6)^{1/2} \frac{1}{\sqrt{\pi}} \frac{2^{8g-6-2k}}{3^{3g-4+k}}.
\end{equation}
On the other hand
\begin{equation} \label{eq:simplify2}
\frac{(6g-7)!}{(6g-7+p_1+\ldots+p_r+r)!}
\sim
\frac{1}{(6g-6)^{p_1 + \cdots + p_r + r}}.
\end{equation}
Replacing~\eqref{eq:simplify1} and~\eqref{eq:simplify2} in~\eqref{eq:asymp1} we obtain
\[
\begin{split}
M_p(\tilde{L}^{(g,m,\kappa)*})
& \sim
\frac{1}{2 \cdot (6g-6)^{p_1 + \cdots + p_r + r - 1/2}} \cdot \sqrt{\frac{m+1}{m}} \cdot \sqrt{\pi}  \\
& \qquad \cdot \sum_{k=1}^{\kappa \frac{\log(6g-6)}{2}}
\frac{1}{(k-r)!}
\sum_{\substack{(j_1,\dots,j_k)\in\NN^k \\ j_1 + \dots + j_k = 3g-3}}
\prod_{i=1}^k \frac{\zeta_m(2j_i)}{2j_i}
\prod_{i=1}^r \frac{(2j_i+p_i)!}{(2j_i-1)!}
\end{split}
\]
which is the announced formula.
\end{proof}
\subsection{Asymptotic expansion of a related sum}
Let $\theta = (\theta_i)_{i \geq 1}$ be a sequence of non-negative real numbers and
let $p = (p_1, \ldots, p_r)$ be a non-negative integral vector.
This section is dedicated to the asymptotics in $n$ of the numbers
\begin{equation} \label{eq:Sthetan}
S_{\theta,p,n}
\coloneqq
\sum_{k=r}^{\infty} \frac{1}{(k-r)!} \sum_{\substack{(j_1,\dots,j_k) \in \NN^k \\ j_1+\cdots +j_k = n}}
    \prod_{i=1}^k \frac{\theta_i}{2 j_i} \prod_{i=1}^r \frac{(2 j_i + p_i)!}{(2j_i - 1)!}.
\end{equation}
which should be reminiscent of the formula from Theorem~\ref{thm:LgMomentsAsymptotics}.

\begin{df}
Let $\theta = (\theta_j)_{j \geq 1}$ be non-negative real numbers and let
$g_\theta(z)$ be the formal series
\begin{equation} \label{eq:gtheta}
g_\theta(z) \coloneqq \sum_{j \geq 1} \theta_j \frac{z^j}{j}.
\end{equation}
We say that $\theta$ is \emph{admissible} if the function $g_\theta(z)$
\begin{itemize}
\item converges in the open disk $D(0,1)\subset\CC$ centered at $0$ of radius $1$,
\item $g_\theta(z) + \log(1-z)$ extends to a holomorphic function on $D(0,R)$
with $R > 1$.
\end{itemize}
\end{df}

\begin{thm} \label{thm:transfer}
Let $\theta = (\theta_k)_{k \geq 1}$ be admissible, then as $n \to \infty$ we have
\begin{equation} \label{eq:SthetaAsymptotics}
S_{\theta,p,n} \sim \sqrt{\frac{e^{\beta}}{2}} \cdot \frac{p_1! \cdots p_r!}{2^{r-1}} \frac{n^{p_1 + \cdots + p_r + r - 1/2}}{\varGamma(p_1 + \cdots + p_r + r + 1/2)}
\end{equation}
where $\beta$ is the value at $z=1$ of $g_\theta(z) + \log(1-z)$.
\end{thm}

The following is essentially~\cite[Lemma~3.8]{DGZZ20} that we reproduce
for completeness.
\begin{lem}
For $m \in \NN \cup \{+\infty\}$, let
\[
g_m(z) \coloneqq \sum_{j \ge 1} \zeta_m(2j) \frac{z^j}{j}.
\]
Then $g_m(z)$ is summable in $D(0,1)$ and $g_m(z) + \log(1-z)$ extends
to a holomorphic function on $D(0,4)$.
In particular the sequence $(\zeta(2j))_{j \ge 1}$ is admissible.
Moreover
$(g_m(z) + \log(1-z))|_{z=1} = \log \left( \frac{2m}{m+1} \right)$.
\end{lem}

\begin{proof}
Since $\zeta_m(2j)$ is bounded uniformly in $j$, the series converges in $D(0,1)$.
Now, expanding the definition of the partial zeta function $\zeta_m$ and changing
the order of summation we have for $z \in D(0,1)$
\[
g_m(z) = - \sum_{n=1}^m \log \left(1 - \frac{z}{n^2} \right)
\]
and hence
\[
g_m(z) + \log(1-z) = - \sum_{n=2}^m \log \left(1 - \frac{z}{n^2} \right).
\]
The term $\log \left(1 - \frac{z}{n^2} \right)$ defines a holomorphic function
on $D(0,n^2)$. Since
$\left| \log \left(1 - \frac{z}{n^2} \right) \right| \le
\frac{4}{n^2}
\left| \log \left(1 - \frac{z}{4} \right) \right|
$
we have absolute convergence even for $m=+\infty$ and $g_m(z) + \log(1-z)$
defines a holomorphic function in $D(0,4)$.

Now for the value at $z=1$ we obtain
\begin{align*}
(g_m(z) + \log(1-z))|_{z=1} &=
- \sum_{n=2}^m \log \left(1 - \frac{1}{n^2}\right) \\
&= \sum_{n=2}^m \left(2 \log(n) - \log(n-1) + \log(n+1 ) \right) \\
&= \log \left( \frac{2m}{m+1} \right).
\end{align*}
This completes the proof.
\end{proof}

\begin{cor} \label{cor:zetamAsymptotics}
Let $m \in \NN \cup \{+\infty\}$.
For $\theta = (\zeta_m(2i))_{i \geq 1}$ we have
\[
S_{\theta,p,n} \sim
\sqrt{\frac{m}{m+1}}
\cdot
\frac{p_1! \cdots p_r!}{2^{r-1}} \frac{n^{p_1 + \cdots + p_r + r - 1/2}}{\varGamma(p_1 + \cdots + p_r + r + 1/2)}.
\]
\end{cor}

For a non-negative integer $p$ we define the differential operator on $\CC[[z]]$
by
\[
D_p(f) \coloneqq z \frac{d^{p+1}}{dz^{p+1}} (z^p f).
\]

We start with some preliminary lemmas.
\begin{lem} \label{lem:coeffExtraction}
Let $\theta = (\theta_i)_i$ and $g_\theta(z)$ as in Theorem~\ref{thm:transfer}.
Let $p = (p_1, \ldots, p_r)$ be a tuple of non-negative integers and let
\begin{equation} \label{eq:Gthetap}
G_{\theta,p}(z) \coloneqq \exp\left(\frac{1}{2} g_\theta(z^2)\right) \prod_{i=1}^r D_{p_i}\left( \frac{1}{2} g_\theta(z^2) \right).
\end{equation}
Then, for any $n \geq 0$ we have $[z^{2n}] \, G_{\theta,p}(z) = S_{\theta,p,n}$ where $[z^{2n}]$ is the coefficient extraction operator and $S_{\theta,p,n}$ is the sum in~\eqref{eq:Sthetan}
\end{lem}

\begin{proof}
Let us first note that $\frac{1}{2} g_\theta(z^2) = \sum_i \theta_i \frac{z^{2i}}{2i}$. We aim
to compute the expansion of $D_p \left( \frac{1}{2} g_\theta(z^2) \right)$. By linearity, it is
enough to compute a single term and we have
\[
D_p \left( z^{2j} \right) = z \frac{d^{p+1}}{dz^{p+1}} (z^{2j+p}) = \frac{(2j + p)!}{(2j-1)!} z^{2j}.
\]
Hence
\[
D_p \left( \frac{1}{2} g_\theta(z^2) \right)
=
\sum_{j=1}^{\infty} \frac{(2j+p)!}{(2j-1)!} \theta_i \frac{z^{2j}}{2j}.
\]
The lemma follows by expanding the exponential.
\end{proof}

\begin{lem} \label{lem:DpOfLog}
For any $p \geq -1$ we have
\[
\frac{1}{p!} D_p \left( - \log(1 \pm z) \right)
= \frac{1}{(1 \pm z)^{p+1}} - 1.
\]
\end{lem}
\begin{proof}
By Leibniz's rule,
\begin{align*}
\frac{z}{p!} \frac{d^{p+1}}{dz^{p+1}}\left(z^p \log\frac{1}{1-z}\right)
&= \frac{z}{p!} \sum_{i=0}^{p} \binom{p+1}{i} p(p-1)\cdots (p-i+1) z^{p-i} \frac{(p-i)!}{(1-z)^{p+1-i}} \\
&= -1 +\sum_{i=0}^{p+1} \binom{p+1}{i} \left(\frac{z}{1-z}\right)^{p-i+1} \\
&= -1 + \left(1 + \frac{z}{1-z}\right)^{p+1} \\
&= -1 + \frac{1}{(1-z)^{p+1}}.
\end{align*}
The proof for $- \log(1 + z)$ is similar.
\end{proof}

\begin{proof}[Proof of Theorem~\ref{thm:transfer}]
By Lemma~\ref{lem:coeffExtraction}, the sum $S_{\theta,p,n}$ is the coefficient in front of $z^{2n}$ of $G_{\theta,p}$.
By the conditions in the statement, $g_\theta(z^2) = - \log(1-z^2) + \beta + r_\theta(z)$ where $r_\theta(z)$ is holomorphic on $D(0,\sqrt{R})$ and $r_\theta(1) = 0$.
Using Lemma~\ref{lem:DpOfLog}, we deduce that for any $p \geq 0$ we have
\[
D_p \left( g_\theta(z^2) \right)
=
\frac{p!}{(1-z)^{p+1}} + \frac{p!}{(1+z)^{p+1}} + r_{\theta,p}(z)
\]
where $r_{\theta,p}$ is holomorphic in $D(0,\sqrt{R})$.
We deduce that $G_{\theta,p}(z)$ is meromorphic in $D(0,\sqrt{R}) \setminus [1,\sqrt{R})$
and satisfies as $z \to 1$
\begin{align*}
    G_{\theta,p}(z) &= \exp\left(- \frac{1}{2} \left(\log(1-z) + \log(2) - \beta \right) + O(1-z) \right)
\prod_{i=1}^r \frac{p!}{2} \left(\frac{1}{(1-z)^{p_i+1}} + O(1) \right) \\
&=
\sqrt{\frac{e^\beta}{2}} \cdot \frac{p_1!\cdots p_r!}{2^r} \frac{1}{(1-z)^{p_1+\cdots+p_r+r+1/2}} \, (1 + o(1)).
\end{align*}
Similarly, as $z \to -1$ we have
\[
G_{\theta,p}(z)
=
\sqrt{\frac{e^\beta}{2}} \cdot \frac{p_1!\cdots p_r!}{2^r} \frac{1}{(1+z)^{p_1+\cdots+p_r+r+1/2}} (1 + o(1)).
\]
Now using~\cite[Theorem VI.5]{FS09}, we obtain
\[
[z^{2n}] \, G_{\theta,p}(z)
\sim 2 \cdot
\sqrt{\frac{e^\beta}{2}} \cdot \frac{p_1! \cdots p_r!}{2^r} \cdot \frac{(2n)^{p_1 + \cdots + p_r + r - 1/2}}{\varGamma(p_1+\cdots+p_r+r+1/2)}.
\]
This completes the proof.
\end{proof}

\subsection{Truncation estimates}
Recall that Theorem~\ref{thm:LgMomentsAsymptotics} provided an expression for the moment $M_p(\tilde{L}^{(g,\kappa)*})$
which involves a sum which is a truncated version of $S_{\theta,p,n}$ from~\eqref{eq:Sthetan}. In this
section, we show that the difference between $S_{\theta,p,n}$ and its truncation is negligible
compared to the asymptotics of Theorem~\ref{thm:transfer}.

\begin{thm} \label{thm:remainder}
Let $\theta$ and $g_\theta(z)$ be as in Theorem~\ref{thm:transfer}.
Then for any real $\kappa > 1$ we have as $n \to \infty$
\begin{equation} \label{eq:sumVSPartialSum}
S_{\theta,p,n} \sim
\sum_{k=r}^{\kappa \frac{\log(2n)}{2}} \frac{1}{(k-r)!} \sum_{\substack{(j_1,\dots,j_k) \in \NN^k \\ j_1+\cdots +j_k = n}}
    \prod_{i=1}^k \frac{\theta_i}{2 j_i} \prod_{i=1}^r \frac{(2 j_i + p_i)!}{(2j_i - 1)!}.
\end{equation}
\end{thm}

Bounding the coefficient in a Taylor expansion is a standard tool in asymptotic
analysis as the ``Big-Oh transfer''~\cite[Theorem VI.3]{FS09}. However, in our
situation we need to bound the $n$-th Taylor coefficient of a function $f_n$
that depends on $n$. To do so, we track down the dependencies on the
functions inside the transfer theorem.

\begin{lem}[{\cite[Lemma 4.4]{DGZZ20}}]\label{lem:DGZZ20.4.4}
Let $\lambda$ and $x$ be positive real numbers. We have,
\[
    \sum_{k=\lceil x\lambda \rceil}^{\infty} \frac{\lambda^k}{k!}
    \leq \exp(-\lambda(x\log x -x)).
\]
\end{lem}

\begin{lem} \label{lem:partialSumUpperBound}
Let $g(z)$ be a holomorphic function on $D(0,R) \setminus ([1,R) \cup [-1,-R))$ such that, as $z \to \pm 1$ we have
\[
h(z) = - \frac{1}{2} \log(1-z^2) + O(1).
\]
Fix a real $\kappa > 1$, and non-negative integers $p$ and $q$.
For $n \geq 1$ let
\[
f_n(z)
= \frac{1}{(1-z)^p (1+z)^q} \sum_{k = \lfloor \kappa \frac{\log n}{2} \rfloor}^{\infty} \frac{h(z)^k}{k!}.
\]
Then, we have as $n \to \infty$
\[
[z^n] \, f_n(z)
= O\left( n^{\max\{p,q\}-1- \frac{1}{2} (\kappa\log\kappa - \kappa)} \right).
\]
\end{lem}
\begin{proof}
    Let $0 < \eta < R-1$ and $0 < \phi < \pi/2$ and define the contour $\gamma$ as the union $\sigma_+ \cup \sigma_- \cup \lambda_\nearrow \cup \lambda_\nwarrow \cup \lambda_\searrow \cup \lambda_\swarrow \cup \lambda_\nearrow \cup \Sigma_+ \cup \Sigma_-$ with
    \begin{align*}
        \sigma_+ & = \{z : |z-1| = 1/n, \ |\arg(z-1)|\geq \phi\}, \\
        \sigma_- & = \{z : |z+1| = 1/n, \ |\arg(z-1)|\geq \phi\}, \\
        \lambda_\nearrow & = \{z : |z-1|\geq 1/n,\ |z| \leq 1 + \eta, \arg(z-1) = \phi\}, \\
        \lambda_\nwarrow & = \{z : |z-1|\geq 1/n,\ |z| \leq 1 + \eta, \arg(z-1) = -\phi\}, \\
        \lambda_\searrow & = \{z : |z+1|\geq 1/n,\ |z| \leq 1 + \eta, \arg(z-1) = \pi - \phi\}, \\
        \lambda_\swarrow & = \{z : |z+1|\geq 1/n,\ |z| \leq 1 + \eta, \arg(z-1) = -\pi + \phi\}, \\
        \Sigma_+ & = \{z : |z| = 1+\eta,\ \arg(z-1) \geq \phi,\ \arg(z+1) \leq \pi - \phi\}, \\
        \Sigma_- & = \{z : |z| = 1+\eta,\ \arg(z-1) \leq - \phi,\ \arg(z+1) \geq -\pi + \phi\}.
    \end{align*}
    See Figure~\ref{fig:contour} for a picture of $\gamma$.
    Sine $f_n$ is holomorphic on $D(0,R) \setminus ([1,R) \cup [-1,-R))$, we have the Cauchy's residue theorem for its coefficients
    \begin{equation} \label{eq:Cauchy}
        [z^n] \, f_n(z)
        =
        \frac{1}{2\pi i}\int_{\gamma} \frac{f_n(z)}{z^{n+1}} \, dz.
    \end{equation}

    \begin{figure}[ht]
        \centering \includegraphics{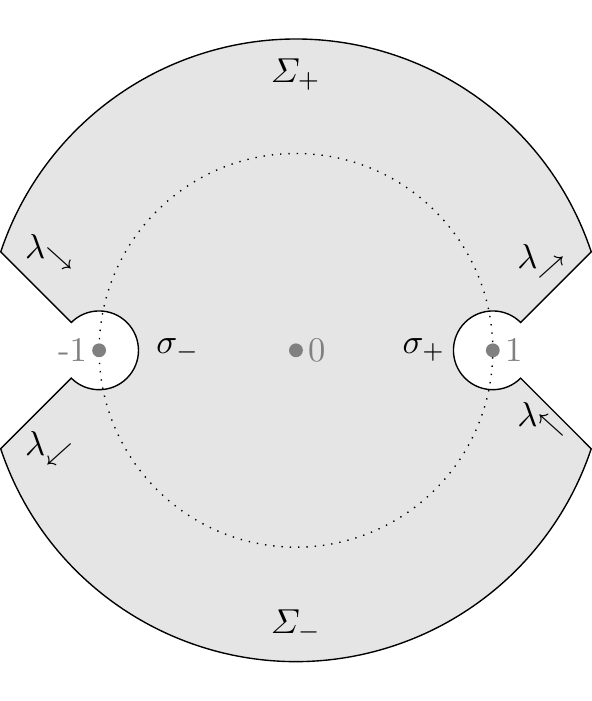}
        \caption{The contour $\gamma$.}
        \label{fig:contour}
    \end{figure}
    Taking absolute values in~\eqref{eq:Cauchy} we obtain
    \begin{equation} \label{eq:CauchyUpperBound}
        \left| [z^n] \, f_n(z) \right|
        \leq \frac{1}{2\pi} \int_\gamma \frac{|dz|}{|z|^{n+1}} \frac{1}{|1-z|^p |1+z|^q} \sum_{k=\lfloor \kappa \frac{\log n}{2} \rfloor}^{\infty} \frac{|h(z)|^k}{k!}.
    \end{equation}
    The proof proceeds by analyzing the right-hand side in~\eqref{eq:CauchyUpperBound} for each piece of the contour $\gamma$.

    Let us start with the small arc of the circle $\sigma_+$.
    The change of variables $z = 1-e^{i \theta}/n$ yields
    \[
        \left|\frac{1}{2\pi i} \int_{\sigma_+} \frac{f_n(z)dz}{z^{n+1}} \right|
        \leq
        \frac{n^{p-1}}{2\pi \cdot (R+1)^q} \int_{-\pi + \phi}^{\pi-\phi} \frac{d\theta}{|1-e^{i\theta}/n|^{n+1}} \sum_{k=\lfloor \kappa \frac{\log n}{2} \rfloor}^{\infty} \frac{|h(1-e^{i\theta}/n)|^k}{k!}.
    \]
    First $h(1-e^{i \theta}/n) = \frac{\log(n)}{2} + O(1)$ uniformly in $\theta$.
    Hence, by Lemma~\ref{lem:DGZZ20.4.4}, uniformly in $\theta$ as $n \to \infty$ we have
    \[
        \sum_{k=\lceil \kappa \frac{\log n}{2} \rceil}^{\infty} \frac{|h(z)|^k}{k!} \leq \exp \left(- (\kappa\log \kappa - \kappa) \cdot \frac{\log n + O(1)}{2} \right)
        = O \left( n^{-\frac{1}{2} (\kappa\log \kappa - \kappa)} \right).
    \]
    Since $\frac{1}{|1-e^{i\theta}/n|^{n+1}}$ is uniformly bounded in $n$,
    \[
        \left|\frac{1}{2\pi i} \int_{\sigma_+} \frac{f_n(z)\, dz}{z^{n+1}} \right|
        = O\left( n^{p-1-\frac{1}{2} (\kappa\log \kappa - \kappa)} \right).
    \]
    Similarly,
    \[
        \left|\frac{1}{2\pi i} \int_{\sigma_-} \frac{f_n(z)\, dz}{z^{n+1}} \right|
        = O\left( n^{q-1- \frac{1}{2} (\kappa\log \kappa - \kappa)} \right).
    \]

    Let us now consider the case of $\lambda_\nearrow$. Let $r$ be the positive solution of the equation $|1 + re^{i\phi}| = 1+\eta$. Perform the change of variable $z = 1+e^{i\phi} \cdot t/n$, we have
    \begin{align*}
        & \left| \frac{1}{2\pi i}\int_{\lambda_\nearrow} \frac{f(z)}{z^{n+1}}\, dz \right| \\
        & \qquad \leq \frac{n^{p-1}}{2\pi \cdot (R+1)^q} \int_1^{nr} dt \cdot t^{-p} \left| 1+ e^{i\phi}t/n \right|^{-n-1} \sum_{k=\lfloor \kappa \frac{\log n}{2} \rfloor}^{\infty} \frac{|h(1 + e^{i\phi} t/n)|^k}{k!}.
    \end{align*}
    For $n$ large enough and uniformly in $t$, $|h(1+e^{i\phi} t/n)| = \frac{\log(n)}{2} + O(1)$.
    Lemma~\ref{lem:DGZZ20.4.4} gives
    \[
    \sum_{k=\lfloor \kappa \frac{\log n}{2} \rfloor}^{\infty} \frac{|h(1+e^{i/\phi}t/n|^k}{k!}
    =
    O(n^{-\frac{1}{2} (\kappa\log\kappa - \kappa)}).
    \]
    From the boundedness of $|1+e^{i\phi} t/n|^{-n-1}$ it follows that
    \[
    \int_1^{nr} t^{-p} |1+e^{i\phi}\cdot t/n|^{-n-1}\, dt
    = O(1),
    \]
    and therefore
    \[
        \left| \frac{1}{2\pi i} \int_{\lambda_\nearrow} \frac{f(z)dz}{z^{n+1}} \right|
        =
        O\left( n^{-\frac{1}{2} (\kappa\log\kappa - \kappa)} \right).
    \]
    The same estimate is valid for the integral along the other three segments $\lambda_\nwarrow$, $\lambda_\searrow$, and $\lambda_\swarrow$.

    For the two large demi-circles $\Sigma_+$ and $\Sigma_-$, we have
    \[
        \left| \frac{1}{2\pi i} \int_{\Sigma_+} \frac{f(z) dz}{z^{n+1}} \right|
        \leq \frac{1}{2\pi} \cdot \eta^{\, p+1/2} \cdot (1+\eta)^{-n-1} \cdot 2\pi(1+\eta)
        = \frac{\eta^{p+1/2}}{(1+\eta)^{n+1}}
    \]
    which decreases exponentially fast.

    We conclude the proof by combining the above estimates.
\end{proof}

\begin{proof}[Proof of Theorem~\ref{thm:remainder}]
Similarly to Lemma~\ref{lem:coeffExtraction}, if we write
\[
G_{n,\theta,p}(z) \coloneqq \sum_{k = \lfloor \frac{\log(n)}{2} \rfloor}^\infty \frac{(\frac{1}{2} g_\theta(z^2))^k}{k!} \prod_{i=1}^r D_{p_i}\left( \frac{1}{2} g_\theta(z^2) \right).
\]
then $[z^{2n}] \, G_{n,\theta,p}$ is the complement of the partial sum in the right hand side of~\eqref{eq:sumVSPartialSum}. Following the proof of Theorem~\ref{thm:transfer} we obtain as $z \to 1$
\[
G_{n,\theta,p}(z) =
\sum_{k = \lfloor \frac{\log(2n)}{2} \rfloor}^\infty
\frac{(\frac{1}{2} g_\theta(z^2))^k}{k!}
\prod_{i=1}^r \frac{p!}{2} \left(\frac{1}{(1-z)^{p_i+1}} + O(1) \right)
\]
where the $O(1)$ is uniform in $n$ (it only depends on $g_\theta(z)$).
Applying Lemma~\ref{lem:partialSumUpperBound} we obtain
\[
    [z^{2n}] \, G_{n,\theta,p}(z)
    =
    O \left( (2n)^{p_1+\cdots + p_r + r - 1 - \frac{1}{2} (\kappa\log\kappa - \kappa)} \right).
\]
For $\kappa > 1$ we have $-1-\frac{1}{2}(\kappa \log \kappa - \kappa) < -1/2$ and the above sum is negligible compared to the asymptotics of the full sum $S_{\theta,p,n}$ from Theorem~\ref{thm:transfer}.
\end{proof}

\subsection{Proof of Theorem~\ref{thm:GEM}}
\begin{proof}[Proof of Theorem~\ref{thm:GEM}]
    By Lemma~\ref{lem:momentMethod}, it suffices to prove the convergence of the moments $M_p(\tilde{L}^{(g,m,\kappa)*})$ for all $p=(p_1,\ldots,p_r)$ towards the moments of the $\GEM(1/2)$ distribution that were computed in Lemma~\ref{lem:GEMMoments}.

    Now, Theorem~\ref{thm:LgMomentsAsymptotics}, provides an asymptotic equivalence of $M_p(\tilde{L}^{(g,m,\kappa)*})$ involving the sum
\[
\sum_{k=r}^{\kappa \frac{\log(6g-6)}{2}} \frac{1}{(k-r)!} \sum_{\substack{(j_1,\dots,j_k) \in \NN^k \\ j_1+\cdots +j_k = 3g-3}}
    \prod_{i=1}^k \frac{\zeta_m(2 j_i)}{2 j_i} \prod_{i=1}^r \frac{(2 j_i + p_i)!}{(2j_i - 1)!}.
\]
The asymptotics of the above sum was then obtained from Corollary~\ref{cor:zetamAsymptotics} and Theorem~\ref{thm:remainder}. Namely, the above is asymptotically equivalent to
\[
\sqrt{\frac{m}{m+1}} \cdot
\frac{p_1! \cdots p_r!}{2^{r-1}} \cdot \frac{(6g-6)^{p_1 + \cdots + p_r + r - 1/2}}{\varGamma(p_1 + \cdots + p_r + r + 1/2)}.
\]
Substituting this value in the formula of Theorem~\ref{thm:LgMomentsAsymptotics} we
obtain as $g \to \infty$
\[
M_p(\tilde{L}^{(g,m,\kappa)*}) \sim
\frac{\sqrt{\pi}}{2^r} \cdot
\frac{p_1! \cdots p_r!}{\varGamma(p_1 + \cdots + p_r + r + 1/2)}.
\]
The above is the value of the moments $M_p$ of the distribution $\GEM(1/2)$ from
Lemma~\ref{lem:GEMMoments} as $\theta=1/2$ and $(\theta-1)! = (-1/2)! = \varGamma(1/2) = \sqrt{\pi}$.

Since the convergence of $M_p(\tilde{L}^{(g,m,\kappa)*})$ holds for all $p=(p_1,\ldots,p_r)$, the sequence $\tilde{L}^{(g,m,\kappa)*}$ converges in distribution towards $\GEM(1/2)$.
\end{proof}

\end{document}